\documentclass[a4paper, 11pt]{article}
\usepackage[utf8]{inputenc}

\usepackage{verbatim}
\usepackage{hyperref}
\usepackage{amsfonts}
\usepackage{color}
\usepackage{graphicx}
\usepackage{amsthm}
\usepackage{amsmath,amssymb}
\usepackage{lipsum}
\usepackage{wrapfig}
\usepackage{a4wide}
\usepackage[font={scriptsize}]{caption}
\usepackage[font={scriptsize}]{subcaption}
\usepackage{enumitem}
\setlist[itemize,1]{leftmargin=\dimexpr 26pt}
\newcommand{\RR}{{\mathbb R^2}}

\newcommand{\Hd}{\mathcal H}
\newcommand{\Lb}{\mathcal L}
\newcommand{\T}{\mathcal T}

\newcommand{\D}{\mathcal D}
\newcommand{\F}{\mathcal F}
\newcommand{\I}{\mathcal I}
\newcommand{\Q}{\mathcal Q}

\newcommand{\C}{\mathcal C}

\newcommand{\bxi}{\boldsymbol \xi}
\newcommand{\bet}{\boldsymbol \eta}
\newcommand{\dd}{\,\mathrm d}

\newcommand{\norm}{\boldsymbol \nu}

\newcommand{\bx}{\boldsymbol x}
\newcommand{\by}{\boldsymbol y}

\renewcommand{\div}{\mathrm{div}\,}
\newcommand{\Per}{\mathrm{Per}}

\newcommand{\sgn}{\mathrm{sgn}\,}
\newcommand{\ess}{\mathrm{ess}\,}
\newcommand{\argmax}{\mathrm{arg}\,\mathrm{max}\,}
\newcommand{\argmin}{\mathrm{arg}\,\mathrm{min}\,}
\newcommand{\supp}{\mathrm{supp}\,}

\newcommand{\J}{\mathcal J}
\newcommand{\R}{\mathcal R}
\renewcommand{\S}{\mathcal S}
\newcommand{\dist}{\mathrm{dist}}
\newcommand{\interior}{\mathrm{int}\,}
\newcommand{\z}{\boldsymbol z}
\newcommand{\res}               {\!\!\mathop{\hbox{
                                \vrule height 7pt width .5pt depth 0pt
                                \vrule height .5pt width 6pt depth 0pt}}
                                \nolimits}

\newtheorem{thm}{Theorem}
\newtheorem{lemma}{Lemma}
\newtheorem{defn}{Definition}

\newtheorem{ex}{Example}
\newtheorem*{remark}{Remark}

\def\XXint#1#2#3{{\setbox0=\hbox{$#1{#2#3}{\int}$ }
\vcenter{\hbox{$#2#3$ }}\kern-.6\wd0}}

\author{Micha{\l} {\L}asica$^*$, \quad Salvador 
Moll$^\dagger$, \quad  Piotr B.\;Mucha$^*$ \vspace{5pt}\\
\small $^*$ Institute of Applied Mathematics and
Mechanics, University of Warsaw \\
\small $^\dagger$ Departament d'An\` alisi Matem\` atica, Universitat de
Val\` encia}
\title{Total variation denoising in $l^1$ anisotropy}
\date{\today}

\begin{document}

\maketitle

\begin{abstract}
We aim at constructing solutions to the minimizing problem for the variant of
Rudin-Osher-Fatemi denoising model with rectilinear anisotropy and to the
gradient flow of its underlying anisotropic total variation functional. We 
consider a naturally defined class of functions piecewise constant on 
rectangles ($PCR$). This class forms a strictly dense subset of the space of 
functions of bounded variation with an anisotropic norm. The main result
shows that if the given noisy image is a $PCR$ function, then solutions to both
considered problems also have this property. For $PCR$ data the 
problem of
finding the solution is reduced to a finite algorithm. We discuss some 
implications
of this result, for instance we use it to prove that continuity is preserved by
both considered problems.
\end{abstract}
\smallskip

\begin{center} 
\end{center}


\tableofcontents

\section{Introduction}\label{intro}
In \cite{es-os}, the authors introduced the anisotropic version of the
celebrated model by Rudin, Osher and Fatemi (ROF) \cite{ru-os-fa} of
total variation based noise removal from a corrupted image.
The idea was to substitute the total variation term in the energy functional
\begin{equation}\label{ROF}
\int_{\Omega} |Du|+\frac{1}{2\lambda}\int_{\Omega}(u_0-u)^2\dd \Lb^2,
\end{equation}
by an anisotropic total variation term suitably chosen for a particular given 
image:
\begin{equation}\label{AROF}
\int_{\Omega} |Du|_\varphi
+\frac{1}{2\lambda}\int_{\Omega} (u_0-u)^2\dd\Lb^2.
\end{equation}
Although from the point of view of image processing it is most natural to 
consider the domain $\Omega$ being a rectangle, in principle it can be any open 
bounded set with reasonably regular (e.\,g.\;Lipschitz) boundary, or the whole 
plane $\mathbb R^2$.  The function $|\cdot|_\varphi\colon\mathbb R^2\to 
[0,+\infty[$ encoding the anisotropy is assumed to be convex, positively
1-homogeneous and such that $|\bx|_\varphi> 0$ if $\bx\neq 0$.
Observe that \eqref{AROF} is a generalization of \eqref{ROF} for which
$|\cdot|_\varphi$ is the Euclidean norm, $|\cdot|$. In this case, the 
associated 
\emph{Wulff
shape},
\[W_\varphi:=\{\by\in\mathbb R^2 \colon \by\cdot \bx\leq
|\bx|_\varphi \text{ for all }\bx\in \mathbb R^2\},\]
is exactly the unit ball
with respect to the Euclidean distance. Because of that, minimizers of 
\eqref{ROF} give rise to convex shapes which are smooth (as is the Euclidean
ball). If, instead, $|\cdot|_\varphi$ is a crystalline anisotropy
(in the sense that the Wulff shape is a polygon), then minimizers of
\eqref{AROF} give rise to convex shapes which are compatible with the Wulff
shape, and therefore not smooth anymore.

This new approach has been successfully applied to most of the
classical problems in image processing including denoising (see
\cite{se-ste-teu},\cite{go-os} and \cite{gras-len}), cartoon extraction
\cite{be-bu-dr-ne-ru}, inpainting \cite{cha-se-ste}, deblurring \cite{che-wa-li}
or denoising and deblurring of 2-D bar codes \cite{chos-ge-ob}. In
most of these works, the chosen anisotropy is the $l^1$ norm in the plane;
i.\,e.\;$|\bx|_\varphi=|\bx|_1:=|x_1|+|x_2|$. In this case, the corresponding
Wulff shape is the unit ball with respect to the $l^\infty$ distance; i.\,e.\;a
square.

In the present paper, we focus on the case $|\bx|_\varphi=|\bx|_1$. We give 
an explicit expression for the
minimizer when the corrupted image $u_0$ belongs to the class of functions
piecewise constant on rectangles (as in the case of applications), denoted by
$PCR(\Omega)$ (see
section \ref{sec:rect} for precise definitions). The
minimizer turns out to belong to  $PCR(\Omega)$.

Let us briefly explain the algorithm for construction of the minimizer. 
Given a function $u_0\in PCR(\Omega)$, we consider $G_{u_0}$, the minimal 
grid associated to the level sets of $u_0$ (which are rectilinear polygons, 
precise definitions are given in section \ref{sec:rect}). Then, we construct 
level sets $F_k$ of $u$, starting with highest values of $u$, as follows:
\begin{itemize}
  \item{Step $1$.} Take as $F_1$ the largest minimizer of the following 
\emph{Cheeger quotient} $\J_1$ among all possible rectilinear polygons $E$ 
contained in $\overline{\Omega}$ subordinate 
to $G_{u_0}$: 
\[\J_1(E) = \frac{\Hd^1(\partial E \cap \Omega)-\frac{1}{\lambda}\int_E u_0\, 
\dd\mathcal 
L^2}{\mathcal L^2(E)}.\] 
  \item{Step $k$.} Denote $\check F_k = \bigcup_{i=1}^{k-1} F_i$. If 
$\overline{\Omega}
= \check F_k$, stop. Otherwise, denote by $F_k$ the largest minimizer of the 
following Cheeger quotient $\J_k$ among all possible rectilinear polygons 
contained in 
$\overline{\Omega\setminus \check F_k}$ subordinate to $G_{u_0}$: 
\[\J_k(E) = \frac{\Hd^1(\partial E \cap \Omega \setminus \check F_k)-\mathcal 
H^1(\partial E\cap \partial \check 
F_k)-\frac{1}{\lambda}\int_E u_0\, \dd\mathcal L^2}{\mathcal L^2(E)}.\]
  \end{itemize}
In each rectilinear polygon $F_k$ of resulting decomposition of 
$\Omega$, we define 
\begin{equation}
\label{uJk}
u|_{F_k} = - \lambda \J_k(F_k). 
\end{equation}

In order to prove that $u$ given by this 
algorithm is in fact 
the minimizer, we perform a rather involved mathematical analysis starting 
from the following observation (the Euler-Lagrange equation for \eqref{AROF}): 
$u$ is a minimizer of \eqref{AROF} if and 
only if $\frac{u-u_0}{\lambda}$ belongs to negative subdifferential of the 
energy functional $TV_{\varphi,\Omega}$ on $L^2(\Omega)$ defined  by
\[TV_{\varphi,\Omega}(u) =  \left\{\begin{array}{ll}
                                         \int_{\Omega} |Du|_\varphi &
\text{ if } u \in BV(\Omega), \\
+\infty & \text{ otherwise.} \\
\end{array}\right. \]

This anisotropic energy functional was studied in \cite{moll} in cases that
$\Omega=\mathbb R^N$ or $\Omega$ is a bounded, open, smooth subset of
$\mathbb R^N$, coupled with Dirichlet boundary conditions. The
author characterized the subdifferential as the set of elements of form $\div 
\bxi$ with $\bxi$ satisfying certain conditions (see Theorem 
\ref{ch_subd}). In the case that $\Omega$ is a rectilinear polygon and 
$u_0 \in PCR(\Omega)$, the condition that the 
solution $u$ to
\begin{equation}
 \label{1-min}
 \min_{u \in BV(\Omega)} TV_{1, \Omega}(u) + \frac{1}{2\lambda}\int_\Omega
(u-u_0)^2 \dd \Lb^2
\end{equation}
also belongs to $PCR(\Omega)$ follows from finding a vector field $\bxi\in 
L^\infty(\Omega)$ such that $|\bxi|_\infty \leq
1$ $\Lb^2$-a.\,e., a suitable compatibility condition
on the jump set of $u$ is satisfied (Lemma \ref{charac_subdiff}) and $\div 
\bxi$ 
is piecewise constant on rectangles. Note that once we know that there are 
such $u$ and $\bxi$, and rectilinear 
polygons $F_k$ are ordered level sets of $u$, \eqref{uJk} follows by averaging 
$u = u_0 + \lambda \div \bxi$ 
over each $F_k$.  

We construct the vector field $\bxi$ together with $u$ in Theorem
\ref{minimization} by
minimizing the $L^2$ norm of the divergence over vector fields 
satisfying compatibility conditions. Then, in order to show that 
the divergence is piecewise constant on rectangles, we rely on an auxiliary 
result (Theorem
\ref{cheeger}) in which we prove that a certain anisotropic Cheeger-type
functional on sets (related to the algorithm sketched above) is indeed 
minimized by a 
rectilinear polygon. In the proof of that
result, an important point is that, due to the structure of the Cheeger
quotient, we construct approximate minimizers that belong to
a finite class of rectilinear polygons subordinate to $G_{u_0}$.
As the set of characteristic functions of rectilinear polygons is not compact 
with respect to $TV_{1, \Omega}$, this finiteness is essential.

 On the other hand, our analysis shows that any function piecewise constant on
rectangles belongs to domain of the subdifferential of $TV_{1, \Omega}$ (Lemma 
\ref{charac_subdiff}), and
that this class of functions is preserved by the gradient descent flow of this
functional,
 \begin{equation}
  \label{1-tvflow}
  \left\{\begin{array}{l} u_t \in - \partial TV_{1, \Omega}, \\
u|_{t=0}= u_0.\end{array}\right.
 \end{equation} 
In principle, one could try to 
deduce this result from Theorem
\ref{minimization} by analyzing the discretization of \eqref{1-tvflow} with
respect to time variable (which coincides with a sequence of problems of form
\eqref{1-min} with $\lambda = \Delta t$). Instead, in Theorems \ref{z} and
\ref{solutionform} we do this directly, constructing the vector field that
encodes the solution by means of a number of variational problems. This way we
obtain a finite explicit algorithm for obtaining $u$, with a different
structure than that in Theorem \ref{minimization}. In particular, we use the
semigroup property of solutions to \eqref{1-tvflow}. The results in this case
say that in any of a finite number of intervals between two subsequent time
instances of \emph{merging}, $u_t$ is a fixed function in $PCR(\Omega)$. The 
exact form of $u_t$ is (again) determined by solving a number of (different) 
Cheeger-type problems. In this case, the algorithm is slightly more 
complicated. 
 Given a function $w\in PCR(\Omega)$, we consider again $G_w$ as the minimal 
grid associated to the level sets of $w$. For each level set $Q$, we 
label each part of the boundary  as positive, and we say that 
it belongs to $\partial Q^+$ (resp.\;negative, $\partial Q^-$), if the value of 
$u(t, \cdot)$ inside $Q$ is higher than (resp.\;lower than) the value 
of the level set adjacent to this part of the boundary (thus, we define a 
consistent signature, see subsection \ref{sec:rect}). Then, we produce a 
decomposition of each level set into a family of rectilinear polygons and 
related consisted signature by means of an algorithm similar to the one for the 
minimizer. Finally, to each rectilinear polygon in the decomposition, a constant 
related to the signature is assigned. It is proved that $u_t$ coincides 
with this exact constant (up to next merging time, when the algorithm has to be 
reinitialized). 

We stress that the problem of determining evolution is nontrivial, as at time 
instances of merging, \emph{breaking} may occur along certain line segments, 
leading to expansion of jump set of the solution.

As $\partial TV_{1, \Omega}$ is a monotone operator,  for any
datum $u_0\in L^2(\Omega)$ and a sequence $u_{0,n} \in L^2(\Omega)$,
$n=1,2,\ldots$ such that $u_{0,n}\to u_0$ in $L^2(\Omega)$,
solutions\footnote{Note that both problems \eqref{1-min} and \eqref{1-tvflow}
give rise to one parameter families of functions in $BV(\Omega)$ (in one case
indexed by $\lambda$, in the other --- by $t$). In many cases they coincide, at
least for a range of the parameter (see Theorem \ref{equiv}). If
we refer to \emph{solutions} without precise context, we mean both solutions to
\eqref{1-min} and \eqref{1-tvflow}.} $u_n$ with datum $u_{0,n}$ converge to the
solution with datum $u_0$. It is easy to check that $PCR(\Omega)$ is dense in
$L^2(\Omega)$. In fact, $PCR(\Omega)$ is even strictly
dense in $BV(\Omega)$ (in the sense of seminorm $\int_\Omega |\nabla 
u|_1)$, see \cite[Theorem 3.4]{casaskunischpola}. Therefore, we do not only
give the explicit solution when initial datum belongs to $PCR(\Omega)$, but we
provide an algorithm to compute the solution for any initial corrupted image
with the most natural approximation to it (with functions belonging to the
domain of the subdifferential). 

The idea of a finite dimensional
approximation of problem \eqref{1-min} based on $PCR$ functions is already
present in literature. For instance, in \cite{fitzpatrickkeeling}, the authors
prove that the solutions to \eqref{1-min} where the functional is replaced
with its restrictions (discretizations) to functions piecewise constant on 
finer and finer grids
(and datum $u_0$ replaced by its suitable projections) converge to the solution
to \eqref{1-min}. Our result implies that minimizers to those discrete 
problems are
themselves actual solutions to \eqref{1-min} (with projected datum). 

For a typical example, the space of $PCR$ functions associated 
with the $(M+1) \times (N+1)$ Cartesian grid in a rectangle $\Omega = [0,M] 
\times [0,N]$ is isomorphic to $\mathbb R^{M \times N} = (u_{i,j};\; 
i=1,\ldots, M;\; j = 1, \ldots,N)$. Our result shows that the functional 
$TV_{1, \Omega}$ restricted to this subspace of $PCR(\Omega)$ is equivalent 
to the discrete Ising-type functional 
\begin{equation} 
 \label{ising} 
\sum \{|u_{i,j} - u_{k,l}| \colon 
i,k=1,\ldots,M; \;j,l = 1, \ldots,N;\; |i-k|+|j-m|=1\}.
 \end{equation}
This information means that one can use one of many efficient algorithms, such 
as graph-cut based algorithms (see e.\,g.\;\cite{ChambolleDarbon2009, 
Hochbaum2013} and references therein) devised for minimizing 
discrete functionals involving terms of type \eqref{ising} to obtain the exact 
(up to machine error) solution to \eqref{1-min}. We note here that the 
algorithm proposed by us is of theoretical significance as a tool 
allowing us to prove Theorem \ref{minimization}. 

 Our approach allows us to prove some continuity results about solutions to
\eqref{1-min} and \eqref{1-tvflow}. In particular, if $\Omega$ is
a rectangle or the plane, we prove that if the datum $u_0$ admits a modulus of
continuity of a certain form, then solutions do as well (Theorems
\ref{contpres} and \ref{contpresf}). Analogous results were obtained in the
isotropic case in \cite{CaChNodisc, CaChNoreg}. The method there involves
considering distance between level sets of solution. First, the authors show
that the jump set of solution is contained (up to a
$\Hd^1$-negligible set) in the jump set of initial datum. We point out that
such a result is not true in our case since breaking may appear (see Example
\ref{exbreaking}). Very recently, the continuity result for minimizers of
\eqref{AROF} in convex domains has been proved to hold in the case of general
anisotropies with different methods \cite{mercier}. We still choose to
include continuity results in this paper, mainly because the technique used
here is very much different. The results are basically corollaries of Lemmata
\ref{lemmacont} and \ref{lemmacontf}, which assert non-increasing of maximal
jump on any length scale for $PCR$ data. The Lemmata are of independent interest
from the point of view of computations, as discrete versions of continuity
estimates. Example \ref{excrack} shows
that continuity is not preserved in general, either by \eqref{1-min} or
\eqref{1-tvflow}, if $\Omega$ is not convex.

At this point we note that our results can be seen as generalization of 
observations concerning 1-D problems with total variation. Indeed, in 1-D 
it is easy to see that piecewise constant data are preserved (and consequently, 
 that continuity is preserved), in fact much more is true, as in that 
case $|\nabla u| \leq |\nabla u_0|$ as measures \cite[Corollary 3.2]{briani}. 
The particular simplicity of 1-D case allows for very detailed description of 
solutions (see e.\,g.\;\cite{bonfortefigalli, kielak, ring2000structural}). 

Finally, we remark that our description of solutions shares a connection on the 
formal level to ideas in \cite{burgergilboa} (see also other papers 
referenced there). In \cite{burgergilboa}, the authors show that solutions to 
gradient flow equations of typical discretizations of linear growth functionals 
are piecewise linear in time, and give an explicit expression involving 
suitably defined nonlinear spectral decomposition. Having obtained finite 
reduction of \eqref{1-tvflow} in Theorem \ref{solutionform}, we can use 
it to recover spectral decomposition of $u_0 \in PCR(\Omega)$ with 
respect to \eqref{1-tvflow} via \cite[Conclusion 2 and Theorem 
4.13]{burgergilboa}. 

The plan of the paper is as follows. In section \ref{prelim}, we give some
notation and preliminaries on rectilinear geometry, $BV$ functions,
$L^2$-divergence vector fields as well as the anisotropic total variation and
its gradient
descent flow in $L^2$. In section \ref{auxiliary}, we study an auxiliary 
Cheeger-type problem in
rectilinear geometry. Next, in sections \ref{mini} and \ref{facets} we give
explicit solutions to \eqref{1-min} and \eqref{1-tvflow} in the case that $u_0
\in PCR(\Omega)$, where $\Omega$ is a rectilinear polygon. In section
\ref{cauchy}, we transfer the results to the case $\Omega=\RR$. The 
idealized setting of the whole plane $\RR$ is convenient for discussing 
examples (from the point of view of images, it corresponds to a discrete 
feature set against a uniform background). Since the construction of solutions 
is similar to the previous cases, we only point out
the main differences and state the results. Section \ref{regularity} is devoted
to the study of preservation of moduli of continuity. Finally, in section
\ref{examples}, we show the power of our approach by explicitly computing the
solutions for some data, including the effects of bending and creation of
singularities. After that we end up with some conclusions. 
\section{Notation and preliminaries}\label{prelim}
\subsection{Balls} By $B_\varphi(\bx, r)$ we denote the ball in $\mathbb R^N$ 
with
respect to norm $|\cdot |_\varphi$, centered at $\bx$, of radius $r$. For the
ball with respect to the Euclidean norm, we write simply $B(\bx, r)$. Symbols
$B_\varphi(r), B(r)$ stand for balls centered at the origin.

\subsection{Measures. Lebesgue and Bochner spaces} 
We denote by $\mathcal{L}^N$ and $\mathcal{H}^{N-1}$ the $N$-dimen\-sio\-nal
Lebesgue measure and the $(N-1)$-dimensional Hausdorff measure in $\mathbb R^N$,
respectively. If $A \subset \mathbb R^N$ is a set of positive (possibly 
inifinite) $\Lb^N$ measure, we denote by $L^p(A)$, $1 \leq p \leq \infty$ the 
Lebesgue space of functions integrable with power $p$ with respect to $\Lb^N$. 
On the other hand, if $A \subset \mathbb R^N$ 
has finite $\Hd^{N-1}$ measure (e.\,g.\;$A$ is the boundary of a Lipschitz 
domain), $L^p(A)$ denotes the Lebesgue space of functions integrable 
with power $p$ with respect to $\Hd^{N-1}$. We adopt 
similar notation for spaces $L^p(A, \mathbb R^k)$, $k=2,3, \ldots$ Whenever 
it is clear, we adopt the convention that an equality or inequality between two 
measurable functions holds in the sense of Lebesgue spaces, i.\,e.\;almost 
everywhere with respect to the corresponding (implicitly specified) measure, 
unless otherwise stated.

If $]T_1, T_2[ \subset \mathbb R$ and $X$ is a Banach space, we denote by 
$L^p(]T_1, T_2[, X)$ the usual space of Bochner measurable functions $f \colon 
]T_1, T_2[ \to X$ s.\,t.\;$\int_{T_1}^{T_2} \|f\|_X^p < \infty$. 
By $L_w^p(]T_1, T_2[, X)$ we denote the analogous space of 
weakly measurable functions (see \cite[Chapter I]{barbu}). 

\subsection{Rectilinear polygons}
\label{sec:rect}
We denote by $\R$ the set of closed rectangles in the plane whose sides are 
parallel to the coordinate axes, and by $\I$, the set of all horizontal and 
vertical closed line segments of finite length in the plane.

\begin{wrapfigure}{r}{0.3\textwidth}
  \begin{center}
    \includegraphics[width=0.25\textwidth]{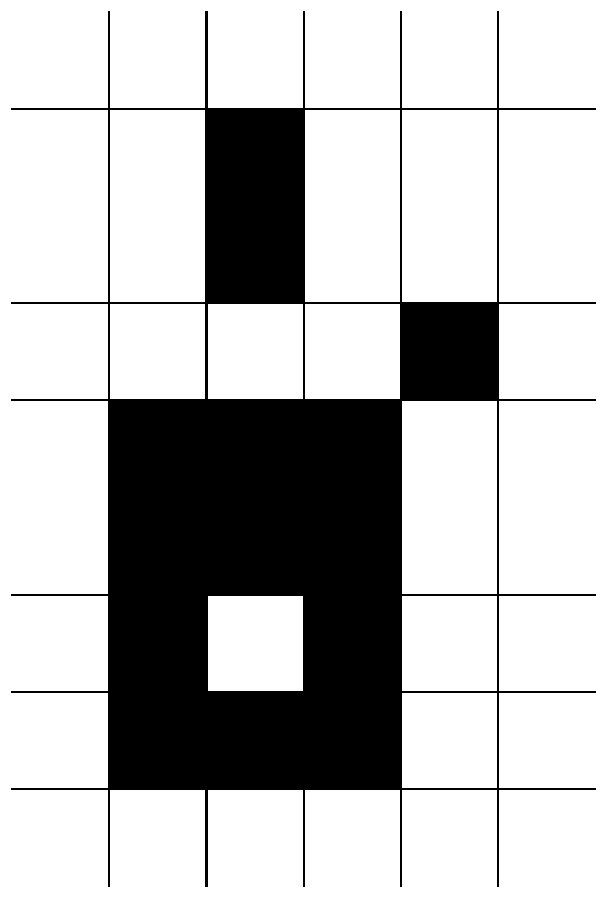}
  \end{center}
  \caption{An example of a rectilinear polygon $F$ and a grid $G$.
There holds $G = G(F)$ and $F \in \F(G)$.}
\end{wrapfigure}

We call $F \subset \mathbb R^2$ a \emph{rectilinear polygon} if $F=\bigcup \R_F$
with a finite $\R_F \subset \R$. We denote by $\F$ the family of all rectilinear
polygons. Similarly, we call $C \subset \RR$ a \emph{rectilinear curve} if $C =
\bigcup \I_C$ with a finite $\I_C \subset \I$. We denote by $\C$ the set of all
recilinear curves.

We call any finite set $G$ of horizontal and vertical lines in the plane a
\emph{grid}. If $F$ is a rectilinear polygon, we denote by $G(F)$ the minimal
grid such that each side of $F$ is contained in a line belonging to $G(F)$.
If $C$ is a rectilinear curve, we denote by $G(C)$ the minimal grid
with the property that there exists $\I_C \subset \I$, $C = \bigcup \I_C$
such that all endpoints of intervals in $\I_C$ are vertices of $G(C)$.

Given a grid $G$, we denote
\begin{itemize}
 \item by $\I(G)$ the set of line segments connecting adjacent vertices of $G$,
 \item by $\R(G)$ the set of rectangles whose sides belong to $\I(G)$,
 \item by $\F(G)$ the set of rectilinear polygons of form $\bigcup \R_F$ with a 
finite non-empty $\R_F \subset \R(G)$.
\end{itemize}
Note that all of the above are finite sets.

It is also convenient to introduce the following notions of partitions of 
rectilinear polygons and signatures for their boundaries.
Let $\Omega$ be a rectilinear polygon. We say that a finite family $\Q$ of
rectilinear polygons with disjoint interiors is a partition of $\Omega$ if
$\Omega = \bigcup \Q$. If $G$ is a grid, we say that a partition $\Q$ of
$\Omega$ is subordinate to $G$ if $\Q \subset \F(G)$. Let $F$ be a rectilinear
polygon and let $G$ be a grid. We say that $(\partial F^+, \partial F^-) \in
\C \times \C$ is a \emph{signature} for $\partial F$ (or for $F$) if $\partial
F^\pm
\subset \partial F$ and $\Hd^1(\partial F^+ \cap \partial F^-) = 0$. We say that
a signature $(\partial F^+, \partial F^-)$ for $\partial F$ is subordinate to
$G$ if both $\partial F^\pm$ are subordinate to $G$. We say that
\[\S\colon \Q \ni Q \mapsto \S(Q) = (\partial Q^+, \partial Q^-) \in \C \times
\C\]
is a \emph{consistent signature} for $\Q$ if
\begin{itemize}
 \item for each $Q \in \Q$, $\S(Q) = (\partial Q^+, \partial Q^-)$ is a
signature for
$\partial Q$ and
 \item for each pair $Q, Q' \in \Q$, if $\bx \in
\partial Q^\pm \cap Q'$ then $\bx \in \partial Q'^\mp$.
\end{itemize}
We say that a consistent signature $\S$ for $\Q$ is subordinate to $G$ if
for each $Q \in \Q$, $\S(Q)$ is subordinate
to $G$.

\begin{figure}[h!]
    \centering
    \begin{subfigure}{0.3\textwidth}
        \includegraphics[width=\textwidth]{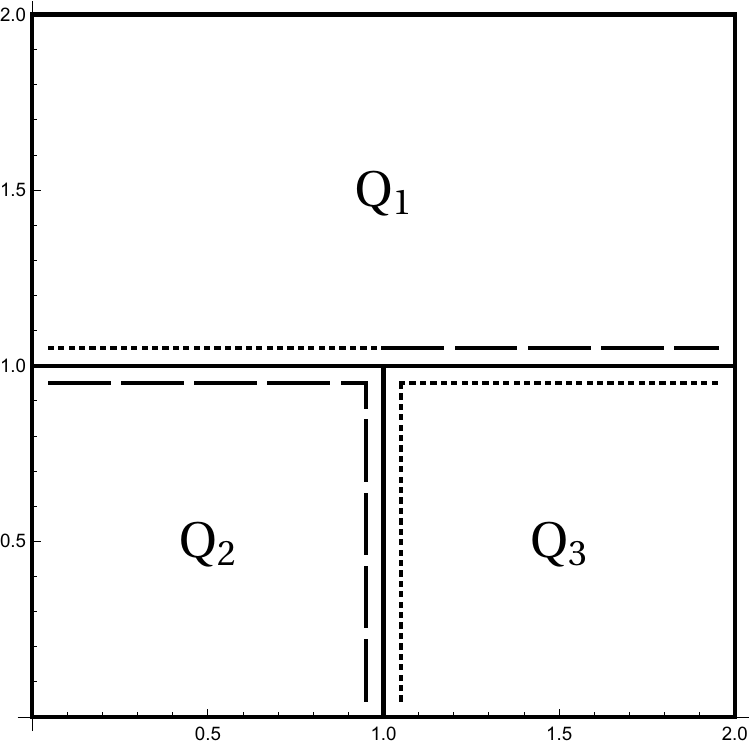}
        \caption{A consistent singature $\S \colon Q_i \mapsto (\partial Q_i^+,
\partial Q_i^-)$ for $\Q = \{Q_1,
Q_2, Q_3\}$. Dashed lines denote $\partial Q_i^+$, dotted lines denote
$\partial Q_i^-$,
$i=1,2,3$. }
        \label{fig:exsign1}
    \end{subfigure}
    \quad
    \begin{subfigure}{0.3\textwidth}
        \includegraphics[width=\textwidth]{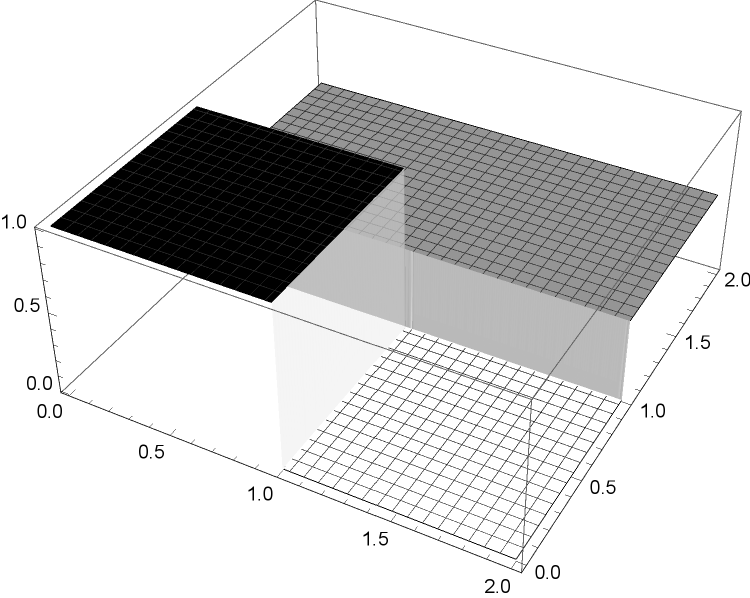}
        \caption{The graph of a function $w \in PCR([0,2]^2)$ such that
$\S = \S_w$.}
        \label{fig:ex1sign3}
    \end{subfigure}
    \quad
    \begin{subfigure}{0.3\textwidth}
        \includegraphics[width=\textwidth]{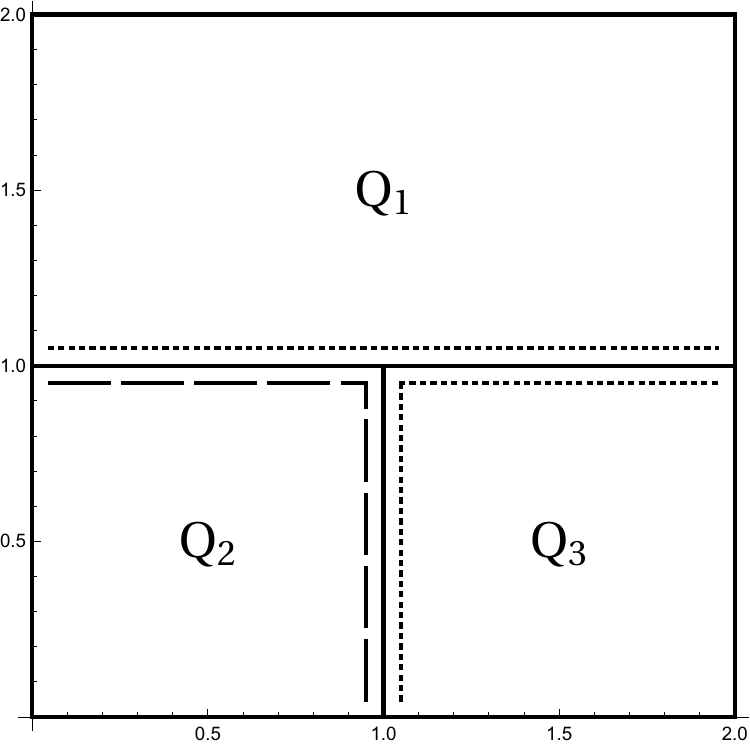}
        \caption{Not a consistent signature.}
        \label{fig:exsign2}
    \end{subfigure}
    \caption{Examples of signatures. }
\label{fig:exsign}
\end{figure}

Now, we  give a precise definition of the class of functions piecewise
constant on rectangles that we will work with. Let $\Omega$ be a rectangle
and let $w \in L^1(\Omega)$. We write
$w \in PCR(\Omega)$ if $w$ has a finite number of level sets of positive
$\Lb^2$ measure, and each one is a rectilinear polygon up to a $\Lb^2$-null
set. We denote the family of level sets of a function $w \in PCR(\Omega)$ by
$\Q_w$. $\Q_w$ is a partition of $\Omega$ in the sense of the definition in
the previous paragraph.

Furthermore, we put $G_w = \bigcup_{Q \in \Q_w}G(Q)$,
$\I_w
= \I(G_w)$, $\R_w = \R(G_w)$, $\F_w = \F(G_w)$. Again, these are all
finite sets.

Given $w \in PCR(\Omega)$ we define the \emph{signature induced by
$w$}, $\S_w \colon \Q_w \ni Q \mapsto (\partial Q^+,
\partial Q^-)$, setting
 \[ \partial Q^+ = \{\bx \in \partial Q \colon \bx \in Q' \in \Q_w,
w|_{Q'} < w|_{Q}\}\]
\[\partial Q^- = \{\bx \in \partial Q \colon \bx \in Q' \in \Q_w, w|_{Q'}
> w|_{Q}\},\]
 for each $Q \in \Q_w$. Here and in many other places we abuse
notation slightly, identifying the constant function $w|_Q$
with its value. The signature induced by
$w$ is a
consistent signature for $\Q_w$ subordinate to $G_w$.

\subsection{Functions of bounded variation and sets of finite perimeter}
 We use standard notation and concepts related to $BV$ functions as in
\cite{ambrosiofuscopallara}; in particular, given $u\in BV(\Omega)$, we
write $\nabla
u\Lb^N$ and $D^s u$ for the absolutely continuous and singular part
of $Du$ with respect to the Lebesgue measure $\Lb^N$, $u^\pm(x)$ for
the lower and upper approximate limits of $u$ at $x \in \Omega$ and
$J_u$ for its jump set, i.\,e.\;the set of points where $u^+ \neq u^-$. Finally,
$\frac{Du}{|Du|}$ denotes the Radon-Nikodym derivative
of $Du$ with respect to its total variation $|Du|$.

The family $PCR(\Omega)$ introduced in the previous subsection is a linear
subspace of $BV(\Omega)$. If $w \in PCR(\Omega)$, we have
\[J_w = \bigcup_{Q
\in \Q_w} \partial Q^+ = \bigcup_{Q \in
\Q_w} \partial Q^-, \qquad w^\pm|_{\partial Q^\pm} = w|_Q\,,\]
where $Q \mapsto (\partial Q^+, \partial Q^-)$ is the
signature induced by $w$. Furthermore, we have
\begin{equation}
 \label{Dw}
 |Dw| = (w^+ - w^-) \Hd^1 \res J_w = \bigcup_{Q \in \Q_w} w|_Q (\Hd^1
\res \partial Q^+ - \Hd^1
\res \partial Q^-).
\end{equation}

Given an open set
$\Omega\subseteq \mathbb R^N$ and a Lebesgue measurable subset $E$ of $\mathbb
R^N$, we say that $E$ has finite perimeter in $\Omega$ if $\chi_{E}\in
BV(\Omega)$ and we write ${\rm Per}(E,\Omega):=|D\chi_{E}|(\Omega)$. If $E$ has 
finite
perimeter in $\mathbb R^N$, we write ${\rm Per}(E):={\rm Per}(E,\mathbb R^N)$.

If $E$ is a set of finite perimeter in $\mathbb R^N$, the jump set of $\chi_E$
is $\Hd^{N-1}$-equivalent to the \emph{reduced boundary} $\partial^*
E$ defined by
the following.\footnote{This is the one point where we choose the notation as
in
e.\,g.\;\cite{giustibook, bnp2} over \cite{ambrosiofuscopallara}.} We say a
point $\bx \in \mathbb R^N$ belongs to $\partial^* E$ if $|D
\chi_E|(B(\bx, \varrho))>0$ for all $\varrho >0$ and quantity $\frac{D
\chi_E(B(\bx, \varrho))}{|D \chi_E|(B(\bx, \varrho))}$ has a limit that belongs
to $\mathbb S^{N-1}$ as $\varrho \to 0^+$. If these conditions hold, we denote
this limit by $\norm^E(\bx)$. There holds
\[ \partial^* E \subset \partial^{\frac{1}{2}} E = \left \{\bx \in \mathbb R^N : 
\lim_{\varrho \to 0^+}\frac{\Lb^N(B(\bx, \varrho) \cap E)}{\Lb^N(B(\bx, 
\varrho))} =\frac{1}{2}\right\},\]
also $\Hd^{N-1}\left(\partial^\frac{1}{2} E \setminus \partial^* E\right) = 0$ 
and $\Hd^{N-1}$-almost every point in $\mathbb R^N$ is either a Lebesgue point 
for $\chi_E$ or belongs to $\partial^* E$.

\subsection{Traces of $L^2$-divergence vector fields}
We consider the space
\begin{equation}
X_\Omega=\left\{\z\in L^{\infty}(\Omega,\mathbb R^N)\,:\
\,\div \z\in L^2(\Omega)\right\}\,.
\end{equation}
In \cite[Theorem 1.2]{anzellotti1}, the weak trace on the boundary of a
bounded Lipschitz domain $\Omega$ of the
normal component of $\z \in X_\Omega$ is defined. Namely, it is proved
that the formula
\begin{equation}
\left\langle [\z,\norm^\Omega]\,,\,\rho\right\rangle:=\int_{\Omega} \rho\,
\div \z \dd \Lb^N + \int_{\Omega}\z\cdot\nabla\rho\dd \Lb^N\,\qquad (\rho\in
C^1\left(\overline{\Omega}\right))
\end{equation}
defines a linear operator $[\cdot,\norm^\Omega]: X_\Omega
\rightarrow
L^\infty(\partial \Omega)$ such that
\begin{equation}
 \label{linftytrace}
\Vert \, [\z,\norm^\Omega] \, \Vert_{L^\infty(\partial
\Omega)} \leq \Vert \z
\Vert_{L^\infty(\Omega)}
\end{equation}
for all $\z \in X_\Omega$ and $[\z,\norm^\Omega]$ coincides with the
pointwise trace of the normal component if $\z$ is smooth.

\subsection{The anisotropic total variation. The anisotropic perimeter}
\label{sec:atv}
We recall here the notion of anisotropic total variation introduced in 
\cite{amarbellettini}. Given an open set $\Omega\subseteq \mathbb R^N$, a norm 
$|\cdot |_\varphi$ on $\mathbb R^N$, and a function $u\in L^2(\Omega)$, we 
define
$$TV_{\varphi, \Omega}(u) :=\sup\left\{\int_\Omega u \, \div \bet \dd \Lb^N
\colon \bet \in C^1_c(\Omega, \mathbb R^N),\ |\bet |^*_\varphi\leq 1
\right\}$$
where $|\cdot |^*_\varphi$ denotes the dual norm associated with $|\cdot
|_\varphi$. This is a proper, lower semicontinuous functional on $L^2(\Omega)$
with values in $[0, \infty]$. We have $TV_{\varphi, \Omega}(u) < + \infty$ iff
$u \in BV(\Omega)$, in which case we use notation $|Du|_\varphi(\Omega) =
TV_{\varphi, \Omega}(u)$. This is an equivalent seminorm on $BV(\Omega)$.

In the analysis of differential equations associated with the functional
$TV_{\varphi, \Omega}$, a crucial role is played by the following result
characterizing the subdifferential of $TV_{\varphi,\Omega}$, whose proof can 
easily be obtained by adapting that of \cite[Theorem 12.]{moll}.

\begin{thm}\label{ch_subd}
 Let $\Omega$ be a bounded Lipschitz domain and let $w \in \D(TV_{\varphi,
\Omega}) = BV(\Omega)$. There holds $v \in -
\partial TV_{\varphi,\Omega}(w)$ iff $v \in L^2(\Omega)$ and there exists $\bxi
\in L^\infty(\Omega, \mathbb R^N)$ such that $v = \div \bxi$ and
\begin{equation}
 \label{domain}
-\int_{\Omega} w \, \div \bxi = \int_\Omega |Dw|_\varphi, \quad |\bxi
|^*_\varphi \leq 1 , \quad
\left[\bxi,\norm^\Omega \right]=0.
\end{equation}
\end{thm}
We denote by $X_{\varphi, \Omega}(w)$ the set of $\bxi \in L^\infty(\Omega, 
\mathbb R^N)$ with $\div \bxi \in L^2(\Omega)$ satisfying  \eqref{domain}.

In the present paper, we are concerned with the case $\varphi = |\cdot|_1$. 
Hence, we have
\[|Du|_1(\Omega)=\int_\Omega |\nabla u|_1 \ dx+\int_\Omega
\left|\frac{Du}{|Du|}\right|_1 d|D^s u|\]
for each $u \in BV(\Omega)$. 

Given a set of finite perimeter $E$ in $\Omega$ (resp.\;in $\mathbb R^N$) we
denote
 ${\rm Per}_1(E, \Omega) = |D \chi_E|_1(\Omega)$ and ${\rm Per}_1(E) = {\rm 
Per}_1(E,\mathbb R^N)$. If $E$ has finite perimeter in $\Omega$, then
\begin{equation}
 \label{perimeterinteral}
 {\rm Per}_1(E, \Omega) = \int_{\partial^* E \cap \Omega} |\norm^E|_1 \dd \Hd^1.
\end{equation}
If $\partial E$ is Lipschitz, we can drop the star in $\partial^*E$, and
$\norm^E$ is the pointwise \mbox{$\Hd^1$-a.\,e.}\;defined outer Euclidean
normal to $E$. Observe that, in the particular case that $F$ is a rectilinear 
polygon, 
\[{\rm Per}_1(F,\Omega)={\rm Per}(F,\Omega).\]

Given $\lambda > 0$, a rectangle $\Omega$ and $u_0 \in BV(\Omega)$ we  consider
the minimization problem \eqref{1-min}. The problem has a unique
solution, which is also the unique solution to the Euler-Lagrange equation
\begin{equation}
 \label{el}
u =u_0 + \lambda \div \z, \qquad \z \in X_{1, \Omega}(u).
 \end{equation}

The following result is an easy corollary of Theorem \ref{ch_subd} for the case 
of $PCR$ functions.
\begin{lemma}
  \label{charac_subdiff} Let $\Omega$ be a rectangle and $w\in PCR(\Omega)$.
Then, $X_{1, \Omega}(w)$ consists of vector fields $\bxi\in
L^\infty(\Omega,\mathbb R^2)$ such that, for any $Q\in \mathcal Q(w)$,
$\bxi|_{Q}\in X_{Q}$  satisfies
 \begin{equation}
 \label{domain1}
\left. [\bxi
,\norm^{Q}]\right|_{\partial Q^\pm} = \mp 1, \qquad|\bxi|_\infty \leq 1 \qquad 
\left.
[\bxi,\norm^{Q}]\right|_{\partial Q\cap \partial \Omega} = 0.
 \end{equation}
 Furthermore, this set is non-empty.
\end{lemma}
\begin{proof}
It is easy to check that with any $\bxi$ satisfying \eqref{domain1},
\eqref{domain} holds. On the other hand, suppose that $\bxi \in X_{1, \Omega}$.
Then, integrating by parts in each $Q \in \Q_w$ on the l.\,h.\,s.\;of the first
item in \eqref{domain} and noting that $\norm^Q|_{\partial Q^\pm} =
\mp \frac{Dw}{|Dw|}$ we get
\[ \int_{J_w} \left[\bxi, \tfrac{Dw}{|Dw|}\right] (w^+ - w^-) \dd \Hd^1 =
\int_{J_w} (w^+ - w^-) \dd \Hd^1. \]
Together with the condition $|\bxi|_\infty \leq 1$ and \eqref{linftytrace} this
implies the first item in \eqref{domain1}.

One way to point out a field in $X_{1,\Omega}(w)$ is to extend it
from $J_w = \bigcup_{Q \in \Q_w} \partial Q$, where one of its components is 
fixed by
\eqref{domain1}, by component-wise linear interpolation.
\end{proof}

\subsection{Anisotropic total variation flows}
\label{sec:atvf}
Another class of natural differential equations associated with functional
$TV_{\varphi,\Omega}$ are anisotropic total variation flows that formally
correspond to Neumann problems
\begin{equation}\label{atv}
\left\{\begin{array}{ll}
 u_t = \div \partial |\cdot|_\varphi(\nabla u) & \text{in } \Omega, \\ 
\left[\partial |\cdot |_\varphi(\nabla u),\norm^\Omega\right]=0 & \text{on } 
\partial \Omega \end{array}\right.
\end{equation} with $\norm^\Omega$ denoting the outer unit normal to $\partial 
\Omega$. In our case, $\varphi=1$. Let us recall the notion of \emph{(strong)
solution} to a general $\varphi$-anisotropic total variation flow, which is an
adaptation of \cite[Definition 4.]{moll} for a bounded Lipschitz domain
$\Omega$.

\begin{defn}
\label{defatvf}
Let $0 \leq T_0 < T_* \leq \infty$. A function $u \in C([T_0,T_*[, 
L^2(\Omega))$ 
is called a strong solution to \eqref{atv} in $[T_0, T_*[$ if $u_t \in 
L^2_{loc}(]T_0, T_*[, L^2(\Omega))$, $u \in L^1_w(]T_0, T_*[, BV(\Omega))$ and 
there exists $\z \in L^\infty(]T_0, T_*[\times\Omega, \RR)$ such that
\begin{equation}
 \label{zeqn}
 u_t = \div \z \quad \text{in } \mathcal D'(]T_0, T_*[\times \Omega ),
\end{equation}
\begin{equation}
 \label{zselect}
 |\z |_{\varphi}^*\leq 1 \quad \text{a.\,e.\;in } ]T_0, T_*[ \times  \Omega,
\end{equation}
\begin{equation}
  [\z(t),\norm^\Omega]=0 \text{ and }
\end{equation}
\begin{equation}
 \label{zint}
 - \int_{\Omega} u\, \div \z \dd \Lb^N = \int_{\Omega} |Du(t,\cdot)|_\varphi
\quad \text{for a.\,e. }t \in ]T_0, T_*[.
\end{equation}
\end{defn}

It can be proved as in \cite[Theorem 11.]{moll} that, given any $u_0 \in
L^2(\Omega)$ and $0\leq T_0<T_*\leq \infty$, there exists a unique strong 
solution $u$ to \eqref{atv} in $[T_0, T_*[$ with $u(T_0,\cdot) = u_0$. Clearly, 
if $0\leq T_0 <T_1<T_2\leq \infty$ and
\[u\in C([T_0,T_2[, L^2(\Omega)) \cap 
L^1_w(]T_0, T_*[, BV(\Omega)), \quad u_t \in L^2_{loc}(]T_0, T_*[, 
L^2(\Omega))\]
is such that $u|_{[T_0, T_1[\times \Omega}$ a strong solution to \eqref{atv} in 
$[T_0, T_1[$ and $u|_{[T_1, T_2[\times \Omega}$ a strong solution to 
\eqref{atv} in $[T_1, T_2[$ then $u$ is a strong solution to \eqref{atv} in 
$[T_0, T_2[$.

In fact, this existential result is a characterization of the Crandall-Ligett
semigroup generated by the negative subdifferential of $TV_{\varphi,\Omega}$. In
the present paper we are concerned with evolution of regular (with respect to
the operator $- \partial TV_{\varphi,\Omega}$) initial data. In such case,
semigroup theory yields following result \cite[Chapter III]{barbu}.
\begin{thm}
\label{divz}
 Let $u_0 \in \D(\partial TV_{\varphi,\Omega})$ and let $u$ be the strong
solution to
\eqref{atv} in $[0, \infty[$ starting with $u_0$. Then, every $\z \in 
L^\infty(]0, \infty[\times\Omega, \mathbb R^N)$ satisfying 
(\ref{zeqn}-\ref{zint}) has a representative (denoted henceforth $\z$) such that
\begin{itemize}
 \item[(1)] in every $t \in [0, \infty[$, $\z(t, \cdot)$ minimizes
 \[\F_\Omega(\bxi) = \int_\Omega (\div \bxi)^2 \dd \Lb^N\]
 in $X_{\varphi, \Omega}(u(t,\cdot))$ and this condition uniquely defines $\div 
\z(t, \cdot)$,
 \item[(2)] the function
\[[0, \infty[ \ni t \mapsto \div \z(t, \cdot) \in L^2(\Omega) \mbox{ \ \ is 
right-continuous,} \]
 \item[(3)] the function
\[[0, \infty[ \ni t \mapsto \|\div \z(t, \cdot)\|_{L^2(\Omega)}  \mbox{ \ \ is 
non-increasing,} \]
 \item[(4)] the function $[0, \infty[ \ni t \mapsto
u(t,\cdot) \in L^2(\Omega)$ is right-differentiable and
\[ \frac{\dd}{\dd t}^+ u(t,\cdot) = \div \z(t, \cdot) \mbox{ \ \ in every $t \in 
[0, \infty[$.} \]
\end{itemize}
\end{thm}

\section{Cheeger problems in rectilinear geometry}\label{auxiliary}

Let $F_0$ be a rectilinear polygon, let $f \in PCR(F_0)$ and let $(\partial
F_0^+, \partial F_0^-)$ be a signature for $\partial F_0$. We denote $G =
G_f \cup G(\partial F_0^+) \cup G(\partial F_0^-)$. 

We introduce a functional $\J_{F_0, \partial F_0^+, \partial F_0^-, f}$
with values in $]-\infty,
+\infty]$ defined on subsets of $F_0$ of positive area given by
\[\J_{F_0, \partial F_0^+, \partial F_0^-, f}(E) = \frac{{\rm Per}_1(E,
\interior F_0) + \Hd^1(\partial^*E \cap \partial F_0^+) - \Hd^1(\partial^*E
\cap \partial F_0^-) - \int_E f \dd \Lb^2}{\Lb^2(E)}\,, \]
if $E$ has finite perimeter and $\J_{F_0, \partial F_0^+, \partial F_0^-,f}(E) = 
+
\infty$ otherwise.
Note that for each measurable $E \subset F_0$ of positive area and finite
perimeter, we have
\[\J_{F_0, \partial F_0^+, \partial F_0^-, f}(E) =
\frac{{\rm Per}_1(E)- \Hd^1(\partial^*E \cap \partial F_0 \setminus
\partial F_0^+) - \Hd^1(\partial^*E \cap \partial F_0^-) - \int_E f \dd
\Lb^2}{\Lb^2(E)}. \]

\begin{lemma}
 \label{rectapprox}
 Let $E \subset F_0$ be a set of finite perimeter with $\Lb^2(E)>0$. Then for 
every $\varepsilon > 0$ there exists a rectilinear polygon $F \in \F(G)$ such
that \[\J_{F_0, \partial F_0^+, \partial F_0^-, f}(F) < \J_{F_0, \partial
F_0^+, \partial F_0^-, f}(E) + \varepsilon.\]
\end{lemma}

\begin{proof}

Throughout the proof, we write for short $\J = \J_{F_0, \partial
F_0^+, \partial F_0^-, f}$. \\

\emph{Step 1. Smoothing \\}
First, given $\varepsilon>0$, we obtain a smooth closed
set $\widetilde E \subset F$ such that $\J(\widetilde E) \leq \J(E) +
\varepsilon$ and $\widetilde E$ does not contain any vertices of $F_0$. For this 
purpose, we adapt the standard method of
smooth approximation of sets of finite perimeter. Namely, we consider
superlevels of smooth functions $\psi_\delta * \chi_E$, $\delta>0$.
Here, $\psi_\delta$ is a standard smooth approximation of unity. Using
Sard's lemma on regular values of smooth functions and the coarea formula for
anisotropic total variation \cite[Remark 4.4]{amarbellettini}, we obtain,
reasoning as in the proof of \cite[Theorem 1.24]{giustibook}, a number $0 < t <
\frac{1}{2}$ and a sequence $\delta_j \to 0^+$ such that
\[\widetilde E_j=\{\psi_{\delta_j} * \chi_E \geq t\}\]
is a smooth set for each $j =1, 2, \ldots$ and
\begin{multline}
\label{limits}
\Lb^2(\widetilde E_j \triangle E) \to 0, \qquad \liminf_{j \to \infty}
{\rm Per}_1(\widetilde E_j) = {\rm Per}_1(E), \\ \Hd^1((\partial^* E) \setminus
\widetilde E_j) \to 0, \qquad \left|\int_{\widetilde E_j} f \dd \Lb^2 - \int_E f 
\dd \Lb^2\right| \to 0.
\end{multline}
Here and in the following we denoted by $\triangle$ the symmetric 
difference. The first two items in \eqref{limits} are covered explicitly in 
\cite{giustibook}. The last one is clear since $f \in L^\infty(F_0)$. It 
remains to justify the third item. Since $\partial^* E \subset 
\partial^\frac{1}{2} E$, for
each $\bx \in \partial^* E$ there is a natural
number $j_0$ such that for every $j>j_0$ there holds $\bx \in \widetilde E_j$.
Thus, as $\Hd^1(\partial^* E) = {\rm Per}(E)$ is finite, the assertion follows 
by
continuity of measures.

Perturbing each $\widetilde E_j$ a little, we can require that $\partial 
\widetilde E_j$ is transverse 
to every line in $G$. Then, $\partial (F_0 \cap 
\widetilde E_j)$ are
piecewise smooth curves and it is visible that all items in \eqref{limits} 
remain true if we
substitute $F_0 \cap \widetilde E_j$ for $\widetilde E_j$. Therefore, for any
given $\varepsilon '>0$ we  choose a number $j$ such that
\begin{multline}
\label{approx1}
 \Lb^2(F_0 \cap \widetilde E_j) > \Lb^2(E) - \varepsilon ', \qquad {\rm 
Per}_1(F_0 \cap
\widetilde E_j) < {\rm Per}_1(E) + \varepsilon ', \\\Hd^1(\partial (\widetilde 
E_j \cap F)\cap \partial F_0 \setminus \partial F_0^+) > \Hd^1( \partial^*
E \cap \partial F_0  \setminus \partial F_0^+) - \varepsilon'\\
\text{and}\quad\Hd^1(\partial (\widetilde E_j \cap F) \cap \partial F_0^-) > 
\Hd^1( \partial^*
E \cap \partial F_0^-) - \varepsilon'.
\end{multline}
Taking $\varepsilon '$ small enough we obtain
\begin{equation}
\label{approx2}
\J(F_0 \cap \widetilde E_j) < \J(E) +
\varepsilon .
\end{equation}
Due to transversality, there is at most a finite number of points where 
the piecewise smooth curve
$\partial(F_0 \cap \widetilde E_j)$ is not infinitely differentiable. Thus, we
can smooth out the set $F_0 \cap \widetilde E_j$ in such a way that
\eqref{approx1}, and consequently \eqref{approx2}, still hold. We denote the
resulting set by $\widetilde E$. Possibly adjusting $\widetilde E$ slighty, 
we can require that it does not contain any vertices of $F_0$. \\

\emph{Step 2. Squaring \\}
Let now $\varepsilon'' >0$. For each $\bx \in \partial \widetilde E$ there is an 
open square $U(\bx)=I(\bx) \times J(\bx)$ of side smaller than $\varepsilon ''$ 
such that $\widetilde E \cap U(\bx)$ coincides
with subgraph of a smooth function $g \colon I(\bx)\to J(\bx)$ or $g\colon 
J(\bx) \to I(\bx)$ and that $U(\bx)$ intersects at most one edge of $F_0$ 
(contained in the supergraph of $g$).
The family $\{U(\bx) \colon \bx \in \partial \widetilde E\}$ is an open cover of
$\partial \widetilde E$. We extract a finite cover $\{U_1, \ldots, U_l\}$,
$l = l(\varepsilon '')$ out of it. We assume that $\{U_1, \ldots, U_l\}$ is
minimal in the sense that none of its proper subsets covers $\partial \widetilde
E$. Let us take $\widehat E_0 = \widetilde E \cup \bigcup_{i = 1}^l W_i$, where
$W_i \subset F_0$  is the smallest closed rectangle containing
$U_i \cap \widetilde E$. The operation of taking a union of $\widetilde E$ with
$W_1$ increases volume while not increasing $l^1$-perimeter. Indeed, denoting
$U_1 = ]a_1, b_1[ \times ]a_2, b_2[$ and assuming without loss of generality
that $\widetilde E \cap U_1$ coincides with the subgraph of a smooth function
$g_1 \colon ]a_1, b_1[ \to ]a_2, b_2[$, we have
\[\int_{W_1 \cap \partial \widetilde E} |\norm^{\widetilde E} |_1 \dd \Hd^1 = 
\int_{]a_1, b_1[} 1 + |g_1'| \dd \Lb^1 \geq |\sup g_1 - g_1(a_1)| + |b_1 - a_1| 
+ |\sup g_1 - g_1(b_1)|\]
and consequently
\[{\rm Per}_1(\widetilde E) = \Hd^1(\partial \widetilde E \setminus W_1) + 
\int_{W_1
\cap \partial \widetilde E} |\norm^{\widetilde E} |_1 \dd \Hd^1 \geq
{\rm Per}_1(\widetilde E \cup W_1).\]
Similarly, we show that taking the union of $\widetilde E \cup W_1$ with $W_2$
does not increase the perimeter, and so on. Furthermore, clearly $\partial
\widetilde E_0 \cap \partial F_0  \setminus \partial F_0^+ \subset \partial 
\widehat E_0 \cap \partial F_0 
\setminus \partial F_0^+$ and $\partial \widetilde E_0  \cap \partial F_0^-
\subset \partial \widehat E_0  \cap \partial F_0^-$. Summing up, we have
\begin{multline}
\Lb^2(\widehat E_0) \geq \Lb^2(\widetilde E), \qquad \Per_1(\widehat E_0) 
\leq \Per_1(\widetilde E), \\ \Hd^1(\partial \widehat E_0 \cap 
\partial F_0 \setminus \partial F_0^+) \geq \Hd^1( \partial \widetilde E \cap 
\partial F_0 \setminus \partial F_0^+), \qquad \Hd^1(\partial \widehat E_0 \cap 
\partial F_0^-) \geq \Hd^1( \partial \widetilde E \cap \partial F_0^-), \\ 
\left| \int_{\widehat E_0} f \dd \Lb^2 - \int_{\widetilde E} f \dd \Lb^2\right| 
\leq 2 \,\ess \max f \cdot l(\varepsilon '') \cdot (\varepsilon'')^2 .
\end{multline}
No point $\bx \in F_0$ is
contained in more than two of $W_1, \ldots, W_l$, and so
\[ l(\varepsilon '') \cdot (\varepsilon '')^2 \leq 2 \Lb^2(\{\bx \colon 
\dist(\bx, \partial \widetilde E) \leq  \varepsilon '' \}) \to 0 \]
as $\varepsilon ''\to 0$. Thus, fixing small enough $\varepsilon''$,
$\J(\widehat
E_0) < \J(E) + \varepsilon$ holds. \\

\begin{figure}[h!]
    \centering
    \begin{subfigure}{0.23\textwidth}
        \includegraphics[width=\textwidth]{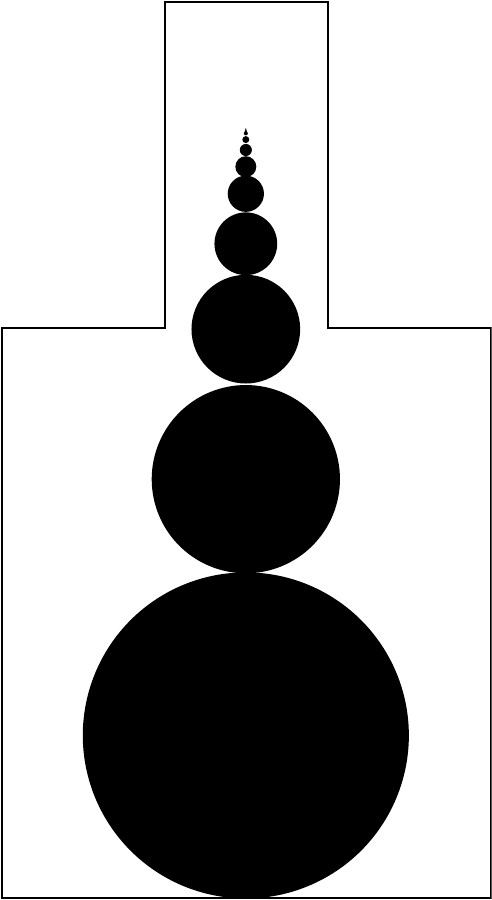}
        \caption{A $BV$ set $E$.}
        \label{fig:bvset}
    \end{subfigure}
    \;
    \begin{subfigure}{0.23\textwidth}
        \includegraphics[width=\textwidth]{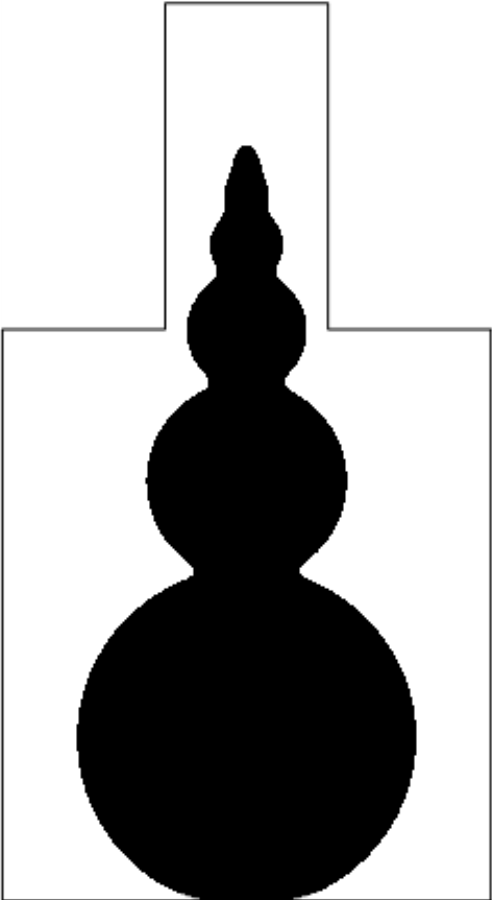}
        \caption{Smoothed set $\widetilde E$.}
        \label{fig:smoothed}
    \end{subfigure}
    \;
    \begin{subfigure}{0.23\textwidth}
        \includegraphics[width=\textwidth]{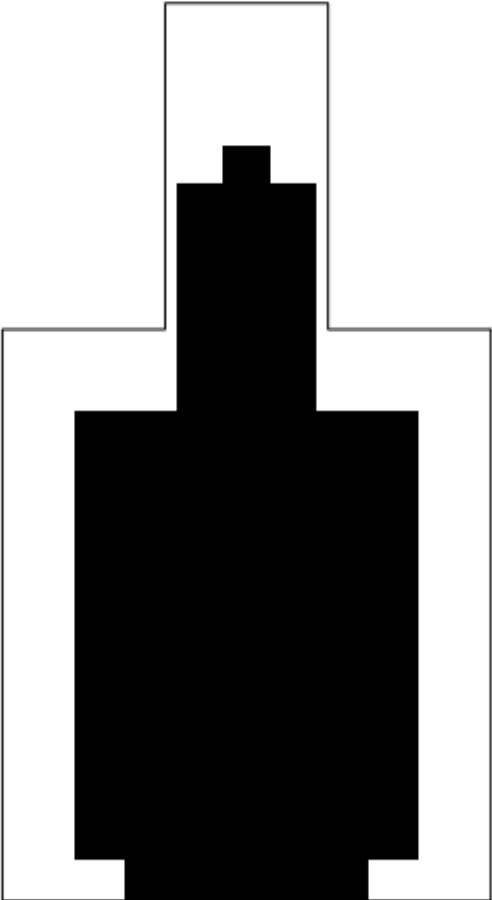}
        \caption{Squared set $\widehat E_0$.}
        \label{fig:squared}
    \end{subfigure}
    \;
        \begin{subfigure}{0.23\textwidth}
        \includegraphics[width=\textwidth]{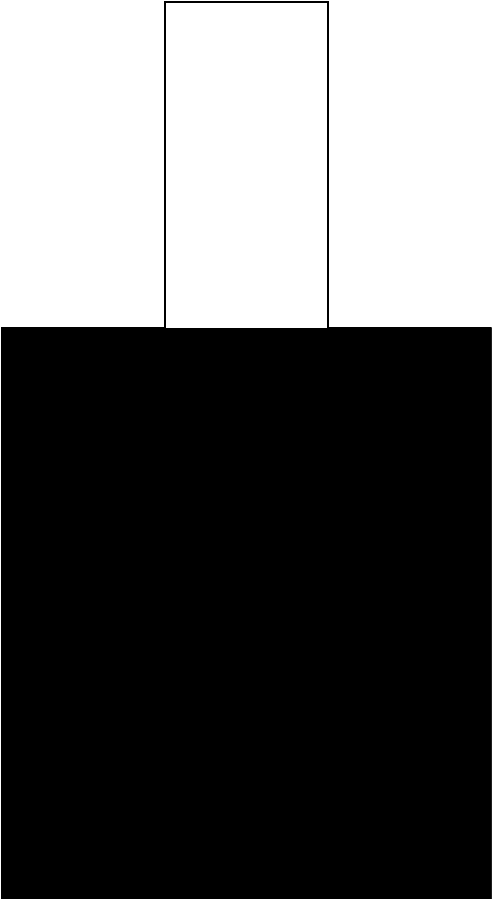}
        \caption{Aligned set $F$.}
        \label{fig:aligned}
    \end{subfigure}
    \caption{The construction in Lemma \ref{rectapprox}
applied to an example set of finite perimeter contained in a
rectilinear polygon (in case $f=0$). }\label{fig:proofscheme}
\end{figure}

\emph{Step 3. Aligning \\}
Take any line $L_0 \subset G(\widehat E_0)$ that is not contained in $G$. We
assume for clarity that $L$ is horizontal, i.\,e.\;$L=\mathbb R \times 
\{y_0\}$, $y_0 \in \mathbb R$. We denote $L_0 \cap \partial \widehat E_0 = 
C_0 \times \{y_0\}$ and observe that $C_0 \subset \mathbb R$ necessarily 
contains 
an interval. 
Let $L_+ = \mathbb R \times \{y_+\}$ and $L_- = \mathbb R \times \{y_-\}$ be 
the lines in $G \cup G(\widehat E_0)$ situated above and below $L_0$ closest 
to $L_0$. We have $C_0 \times [y_-, y_+] \subset F_0$. Let us first 
assume that $\widehat E_0 \neq C_0 \times [y_0, y_+]$, $\widehat E_0 \neq C_0 
\times 
[y_-, y_0]$. For 
$y \in
[y_-, y_+]$, we define
\begin{equation}
\overline \J(y) = \left\{ \begin{array}{ll} 
                           \J(\widehat E_0 \,\triangle\, (C_0 
\times [y_0, y])) & \text{if } y>y_0, \\
\J(\widehat E_0 \,\triangle \,(C_0 
\times [y, y_0])) & \text{otherwise.}
                          \end{array}\right.
\end{equation}
Denoting $\Q_f = \{Q_1, \ldots, Q_n\}$, we observe that $Q_i \cap (C_0 \times 
[y_-, y_+]) 
= C_i \times [y_-, 
y_+]$ with $C_i \subset \mathbb R$ for $i=1, \ldots, n$. This follows 
from the choice of $L_0$ and $L_\pm$. Similarly, each 
one of line segments constituting $\partial C_0 \times [y_-, y_+]$ is 
contained (up to a finite number of points) in $\partial F_+ \cup \interior F$, 
$\partial F_-$ or $\partial F \setminus(\partial F^+ \cup \partial F^-)$.
Therefore, $\overline \J$ is a homography and hence monotone on $]y_-, y_+[$. 

However, $\overline \J$ might be discontinuous at the endpoints of its 
domain. This is only possible if, as $y$ attains $y_+$ (or $y_-$), a pair of 
edges of $\widehat E_0 \,\triangle\, (C_0 \times [y_0, y])$ (resp.\,$\widehat 
E_0 
\,\triangle\, (C_0 \times [y, y_0])$) vanishes, or an edge touches the 
boundary of $F_0$. In either case, there still holds $\lim_{y \to y_\pm} 
\overline \J(y) \geq \overline \J(y_\pm)$. 

Thus, whether
$\overline \J$ is continuous or not, either $\overline \J(y_+) $ or
$\overline \J(y_-)$ (or both) is not larger than $\overline\J(y_0) = 
\J(\widehat E_0)$. In accordance with that, we denote either $\widehat E_1 =
\widehat E_0 \,\triangle\, (C_0 
\times [y_0, y_+])$ or $\widehat E_1 = \widehat E_0 \,\triangle\, (C_0 
\times [y_-, y_0])$
and perform the same argument with $\widehat E_1$ instead of $\widehat E_0$. 

Now, let us go back to the excluded cases and suppose, without loss of 
generality, that $\widehat E_0 = C_0 \times 
[y_0, y_+]$. Then, $\overline \J$ is still a well-defined homography in 
$[y_-, y_+[$ and $\lim_{y \to y_+} \overline \J(y) = + \infty$. Hence, 
$\overline \J(y_-) \leq \overline \J(y_0) = \J(\widehat E_0)$ and we put 
$\widehat E_1 = \widehat E_0 \,\triangle \,(C_0 
\times [y_-, y_0]) = C_0 \times [y_-, y_+]$ and continue the procedure.      

For each $i$, $G(\widehat E_{i+1})$ contains at least one line not 
contained in $G$ less than $G(\widehat E_i)$, so this
procedure terminates in a finite number $s$ of steps and we obtain
$F=\widehat E_s$ whose all edges are contained in $G$ and $\J(F) \leq 
\J(\widehat E_0) < \J(E) + \varepsilon$.
\end{proof}

\begin{thm}
\label{cheeger}
The functional $\J_{F_0, \partial F_0^+, \partial F_0^-,f}$ is bounded from 
below
and is minimized by a rectilinear polygon $F\subset F_0$ such that $F \in
\F(G)$.
\end{thm}
\begin{proof}
 Suppose that $\J_{F_0, \partial F_0^+, \partial F_0^-,f}(E_n)\to -\infty$. 
Then,
due to Lemma \ref{rectapprox} there exist rectilinear polygons $F_n \subset
F_0$, $n=1,2,\ldots$ such that $F_n \in \F(G)$ and $\J_{F_0, \partial F_0^+,
\partial F_0^-,f}(F_n)\to -\infty$, an impossibility.

 Now, consider any minimizing sequence $(E_n)$ of $\J_{F_0, \partial F_0^+,
\partial F_0^-,f}$. By means of Lemma \ref{rectapprox} we find a
minimizing sequence of rectilinear polygons $F_n \in F_0$ such that $F \in
\F(G)$. As the set of such rectilinear polygons is finite, $(F_n)$ has a
constant subsequence $(F_{n_k}) \equiv (F)$. Clearly, $F$ minimizes $\J_{F_0,
\partial F_0^+, \partial F_0^-,f}$.
\end{proof}

Instead of $\J_{F_0, \partial F_0^+, \partial F_0^-,f}$ we can consider
\[\check \J_{F_0, \partial F_0^+,
\partial F_0^-,f}(E) \colon \{ E \subset F_0 \text{ --- measurable s.\,t. } 
\Lb^2(E)>0\} \to [-\infty,
+\infty[\]
defined by
\[\check \J_{F_0, \partial F_0^+,
\partial F_0^-,f}(E) =
\frac{- {\rm Per}_1(E,
\interior F_0) + \Hd^1(\partial^*E \cap \partial F_0^+) - \Hd^1(\partial^*E
\cap \partial F_0^-)- \int_E f \dd \Lb^2}{\Lb^2(E)}, \]
if $E$ has finite perimeter and $- \infty$ otherwise.
Then, noticing that
\[\check \J_{F_0, \partial F_0^+,
\partial F_0^-,f} = - \J_{F_0, \partial F_0^-,
\partial F_0^+, -f}\]
we obtain analogous versions of Lemma \ref{rectapprox}
and Theorem \ref{cheeger}.

\begin{lemma}\label{rectapproxbis}
 Let $E \subset F_0$ be a set of finite perimeter with $\Lb^2(E)>0$. Then for
every $\varepsilon > 0$ there exists a rectilinear polygon $F \in \F(G)$ such
that
\[\check \J_{F_0, \partial F_0^+, \partial F_0^-,f}(F) > \check\J_{F_0,
\partial F_0^+,
\partial F_0^-,f}(E) - \varepsilon.\]
\end{lemma}

\begin{thm}\label{cheegerbis}
 The functional $\check \J_{F_0, \partial F_0^+, \partial F_0^-,f}$ is bounded 
from above
and is maximized by a rectilinear polygon $F\subset F_0$ such that $F \in
\F(G)$.
\end{thm}

\section{The minimization problem for $TV_1$ with $PCR$ datum}\label{mini}
Let $\Omega$ be a rectilinear polygon and let $u_0$ belong to
$PCR(\Omega)$. Given $\lambda>0$, we use Lemma \ref{cheeger} to
prove that the solution $u$ to the problem \eqref{1-min} is encoded in
the partition $\{F_1,\ldots, F_l\}$ of $\Omega$ (partition into level sets 
of $u$) and consistent signature
$F_k \mapsto (\partial F_k^+, \partial F_k^-)$, $k = 1,\ldots, l$ (signature 
induced by $u$),
both subordinate to $G_{u_0}$, produced by the
following 
algorithm:
\begin{itemize}
 \item First, denote by $F_1$ the largest minimizer
of $\J_{\Omega, \emptyset, \emptyset, \frac{u_0}{\lambda}}$, and put $\partial
F_1^+ = \partial F_1 \setminus \partial \Omega$, $\partial F_1^- = \emptyset$.
 \item At $k$-th step, denote $\check F_k = \bigcup_{i=1}^{k-1} F_i$. If $\Omega
= \check F_k$, stop. Otherwise, denote by $F_k$ the largest minimizer
of $\J_{\Omega \setminus \check F_k, \emptyset,\partial \check F_k,
\frac{u_0}{\lambda}}$ and put $\partial F_k^+ = \partial F_k \setminus (\partial
\Omega \cup \partial \check F_k)$, $\partial F_k^- = \partial F_k \cap \partial
\check F_k$.
\end{itemize}
We take the largest 
minimizer, because we want to construct the whole level set in one turn. As 
Cheeger quotients are subadditive, the largest minimizer is the sum of all 
minimizers. Note that the algorithm ends after a finite number of steps, as 
$\check F_k$ is larger than $\check F_{k-1}$ by $F_k$, a non-empty rectilinear 
polygon
subordinate to $G(u_0)$ and contained in $\Omega$. These collectively make a
finite set.
\begin{remark}
 $\partial
F_k^\pm$ are defined in such a way that each $\J_{F_k, \partial F_k^+, \partial
F_k^-, \frac{u_0}{\lambda}}$ is the restriction of  $\J_{\Omega \setminus
\check F_k, \emptyset, \partial \check F_k,
\frac{u_0}{\lambda}}$ to subsets of $F_k$. In particular, $F_k$ is a minimizer
of
 $\J_{F_k, \partial
F_k^+, \partial F_k^-,
\frac{u_0}{\lambda}}$.
\end{remark}

\begin{thm}
 \label{minimization}
Let $F_k$, $(\partial F_k^+, \partial F_k^-)$, $k=1,\ldots,l$ be as above.
Let $u \in PCR(\Omega)$ be given by
 \begin{equation}
 \label{defu}
u|_{F_k} = \frac{1}{\Lb^2(F_k)} \left(\int_{F_k} u_0 \dd \Lb^2 - 
\lambda(\Hd^1(\partial F_k^+) - \Hd^1(\partial F_k^-)) \right)
 \end{equation}
 for $k=1, \ldots, l$. Then $u$ is the solution to \eqref{1-min}.
Furthermore, $\Q_u = \{F_1, \ldots, F_l\}$, $\S_u(F_k)=(\partial
F_k^+, \partial F_k^-)$, $k=1,\ldots,l$.
\end{thm}
\begin{proof}
 We now adapt the reasoning in the proof of Theorem 5 in \cite{thetvf}. For 
given $k=1,\ldots,l$, we consider the functional $\F_k$ defined
 by
 \[\F_k(\bxi) = \int_{F_k} \left(\div \bxi + \frac{u_0}{\lambda}\right)^2 \dd 
\Lb^2\]
 on the set of vector fields $\bet \in X_{F_k}$ satisfying
 \[|\bet|_\infty \leq 1 \qquad \left. [\bet
,\norm^{F_k}]\right|_{\partial F_k^\pm} = \mp 1, \qquad\left. 
[\bet,\norm^{F_k}]\right|_{\partial F_k \cap \partial \Omega} = 0.\]
 We  first prove that any vector field $\bxi \in L^\infty(\Omega, \RR)$, such
that $\bxi|_{F_k}$ satisfies above conditions for each $k=1, \ldots,l$, belongs
to $X_{1, \Omega}(u)$, with $u$ defined by \eqref{defu}. This follows
immediately by Lemma \ref{charac_subdiff} once we know that $F_k 
\mapsto (\partial F_k^+, \partial F_k^-)$, $k=1, \ldots, l$ is the
signature induced by $u$, i.\,e.\;that the inequality
 \begin{multline}
  \label{inequ}
 \frac{1}{\Lb^2(F_{k+1})} \left(\int_{F_{k+1}} u_0 \dd \Lb^2 - 
\lambda(\Hd^1(\partial F_{k+1} \setminus (\partial \Omega \cup \partial \check 
F_{k+1})) - \Hd^1(\partial F_{k+1} \cap \partial \check F_{k+1})) \right) = 
u|_{F_{k+1}} \\ < u|_{F_k} = \frac{1}{\Lb^2(F_k)} \left(\int_{F_k} u_0 \dd 
\Lb^2 - \lambda(\Hd^1(\partial F_k \setminus (\partial \Omega \cup \partial 
\check F_k)) - \Hd^1(\partial F_k \cap \partial \check F_k)) \right),
 \end{multline}
 is satisfied for $k=1, \ldots, l-1$ with $u$ defined in \eqref{defu}. Note 
that we have
 \begin{multline*}
 \Hd^1(\partial F_{k+1} \cap \partial \check F_k) =  \Hd^1(\partial F_{k+1} 
\cap \partial \check F_{k+1}) - \Hd^1(\partial F_{k+1} \cap \partial F_k) \\ 
\text{and} \quad \Hd^1(\partial F_{k+1} \setminus (\partial \Omega \cup \check 
F_k)) = \Hd^1(\partial F_{k+1} \setminus (\partial \Omega \cup \check F_{k+1})) 
+ \Hd^1(\partial F_{k+1} \cap \partial F_k).
 \end{multline*}
 Thus, were \eqref{inequ} not the case, we would have (using the inequality 
$\frac{x_1+x_2}{y_1+ y_2} \leq \frac{x_1}{y_1}$ that holds whenever 
$\frac{x_2}{y_2} \leq \frac{x_1}{y_1}$ for positive numbers $x_1, x_2, y_1, 
y_2$),
 \begin{multline}
  \J_{\Omega\setminus \check F_k, \emptyset, \partial \check F_k, 
\frac{u_0}{\lambda}} (F_k \cup F_{k+1})  \\ \leq \frac{1}{\Lb^2(F_k)+ 
\Lb^2(F_{k+1})}\left(\Hd^1(\partial F_k \cap \partial \check F_k) - 
\Hd^1(\partial F_k \setminus (\partial \Omega \cup \partial \check F_k)) - 
\frac{1}{\lambda}\int_{F_k} u_0 \dd \Lb^2 \right.\\ \left. + \Hd^1(\partial 
F_{k+1} \cap \partial \check F_{k+1}) - \Hd^1(\partial F_{k+1} \setminus 
(\partial \Omega \cup \partial \check F_{k+1})) -\frac{1}{\lambda} 
\int_{F_{k+1}} u_0 \dd \Lb^2 \right)  \\ \leq \J_{\Omega\setminus \check F_k, 
\emptyset, \partial \check F_k, \frac{u_0}{\lambda}} (F_k),
 \end{multline}
 in contradiction with the choice of $F_k$.

Proceeding as in \cite[Proposition 6.1]{bnp1}, we see that $\F_k$ attains a
minimum and for any two minimizers $\bet_1, \bet_2$ we have $\div \bet_1 = \div
\bet_2$ in $F_k$. Let us take any minimizer and denote it $\bxi_{F_k}$.
Arguing as in \cite[Theorem 6.7]{bnp1} and \cite[Theorem 5.3]{bnp2}, 
$\div \bxi_{F_k}
\in L^\infty(F_k) \cap BV(F_k)$. Let $\nu = \J_{F_k, \partial F_k^+, \partial
F_k^-, \frac{u_0}{\lambda}}(F_k)$. We
have
 \[\frac{1}{\Lb^2(F_k)}\int_{F_k} \left(\div \bxi_{F_k} + 
\frac{u_0}{\lambda}\right) \dd \Lb^2 = -
\frac{1}{\Lb^2(F_k)} \left(\Hd^1(\partial F_k^+) - \Hd^1(\partial F_k^-) - 
\frac{1}{\lambda}\int_{F_k} u_0 \dd \Lb^2\right) = - \nu .
\]
 Were $\div \bxi_{F_k} + \frac{u_0}{\lambda}$ not constant in $F_k$, there would 
exist $\mu < \nu$ such that
 \[ A_\mu = \left\{\bx \in F_k \colon - \left(\div \bxi_{F_k}(\bx) + 
\frac{u_0}{\lambda}\right) < \mu\right\} \]
 has positive measure and finite perimeter. Employing \cite[Proposition 
3.5]{bnp2},
 \begin{multline*} - \nu < - \mu < \frac{1}{\Lb^2(A_\mu)}\int_{A_\mu}\left( \div
\bxi_{F_k} + \frac{u_0}{\lambda}\right) \dd \Lb^2 \\= - 
\frac{1}{\Lb^2(A_\mu)}\left(\Per_1(A_\mu, F_k) +
\Hd^1(\partial F_k^+ \cap \partial^* A_\mu) - \Hd^1(\partial F_k^- \cap
\partial^* A_\mu) - \frac{1}{\lambda}\int_{F_k} u_0 \dd \Lb^2 \right) \\ = - 
\J_{F_k, \partial F_k^+, \partial F_k^-,\frac{u_0}{\lambda}}(A_\mu),
 \end{multline*}
 which would contradict that $F_k$ minimizes $\J_{F_k, \partial F_k^+, \partial
F_k^-, \frac{u_0}{\lambda}}$ (see the Remark before the statement of the
Theorem), hence $u_0 + \lambda \div \bxi_{F_k}$ is constant
in $F_k$ and therefore equal to its mean value:
 \[u_0 + \lambda \div \bxi_{F_k} = \frac{1}{\Lb^2(F_k)} \left(\int_{F_k} u_0 \dd 
\Lb^2 - \lambda(\Hd^1(\partial F_k^+) - \Hd^1(\partial F_k^-)) \right) = u,\]
 i.\,e.\;$u$ satisfies the Euler-Lagrange equation \eqref{el}.

 The last sentence of the assertion follows from \eqref{inequ}.
\end{proof}

\begin{remark}
 Instead of considering the minimization problem for $\J$,
one can consider at each step the maximization problem for $\check \J$ (see
Theorem \ref{cheegerbis}).
\end{remark}

\section{The $TV_1$ flow with $PCR$ initial datum}\label{facets}
In what now follows, we are concerned with the identification of the
evolution of initial datum $w \in PCR(\Omega)$, with $\Omega$ a rectilinear 
polygon, under the $l^1$-anisotropic
total variation flow \eqref{1-tvflow}. The result below determines the initial
evolution, prescribing possible
breaking of initial facets.

\begin{thm}
\label{z}
 Let $w \in PCR(\Omega)$ and let $G$ be
any grid such that $\Q_w$ is subordinate to $G$. Then, there exists a
field $\bet \in X_{1, \Omega}(w)$ and, for each $Q \in \Q_w$, a
partition $\T_Q$ of $Q$ and a consistent signature $\S_Q$ for $\T_Q$
subordinate to $G$, $\S_Q\colon F \mapsto (\partial F^+, \partial F^-)$ for $F
\in \T_Q$, such that
 \begin{itemize}
  \item[(1)] $\T = \bigcup_{Q \in \Q_w} \T_Q$ is a partition of $\Omega$
subordinate to
$G$, $\S \colon \T \to \C \times \C$ given by $\S(F) = \S_Q(F)$ for $F \in
\T_Q$ is a consistent signature for $\T$
subordinate to $G$,
  \item[(2)] for each $F \in \T$,
  \[\left.\div \bet\right|_F = - \frac{1}{\Lb^2(F)} \left(\Hd^1(\partial F^+) -
\Hd^1(\partial F^-)\right), \qquad \left. [\bet, \norm^F]\right|_{\partial
F^\pm} = \mp 1.\]
 \end{itemize}
\end{thm}

\begin{proof}
 We fix $Q \in \Q_w$ and produce the partition $\T_Q$ of $Q$ and
consistent signature $\S_Q$ for $\T_Q$ by means of an inductive procedure
analogous to the one in section \ref{mini}. First, by virtue of
Theorem \ref{cheeger}, the functional $\J_{Q, \partial Q^+, \partial Q^-,0}$
attains its minimum value on a rectilinear polygon
$F_1 \in \F(G)$. We define
 \[\partial F_1^- = \partial Q^- \cap \partial F_1
\mbox{ \ \ and \ \ }
\partial F_1^+ = (\partial F_1 \cap \partial Q^+) \cup (\partial F_1 \setminus 
\partial Q)  .\]
Next, in $k$-th step, we put $\check F_k = \bigcup_{j=1}^{k-1} F_j$. If $\check
F_k = Q$ we stop and put $\T_Q = \{F_1, \ldots, F_{k-1}\}$,
$\S_Q(F_j)=(\partial F_j^+, \partial F_j^-)$ for $j = 1,\ldots, k-1$. Otherwise 
we define $F_k$ as any
minimizer of
 \[ \J_{Q\setminus \check F_k, \partial Q^+,
\partial Q^-\cup \partial \check F_k, 0},\]
 and
\[\partial F_k^- = \partial F_k \cap \left( \partial Q^- \cup
\partial \check F_k\right),\]
\[\partial F_k^+ = (\partial F_k \cap \partial Q^+) \cup \left(\partial F_k 
\setminus \left( \partial Q \cup
\partial \check F_k \right)\right).\]

Now, for each $F_k \in \T_Q$, we define $\bet_{F_k}$ as any minimizer of the 
functional $\F_k$ defined on the set of vector fields $\bxi \in L^\infty(F_k, 
\RR)$ satisfying
 \[
\div \bxi \in L^2(F_k),\qquad |\bxi|_\infty \leq 1 \qquad \left. [\bxi, 
\norm^{F_k}]\right|_{\partial F_k^\pm} = \mp 1,
\qquad \left[\bxi, \norm^{F_k}]\right|_{\partial F_k \cap \partial \Omega} =
0
 \]
by $\F_k(\bxi) = \int_{F_k} (\div \bxi)^2 \dd \Lb^2$. As in the proof of Theorem
\ref{minimization}, we prove that $\div \bet_{F_k}$ is constant in each $F_k \in
T_Q$.

Next, we repeat the procedure for the rest of $Q \in \Q_w$ and define
$\bet$ by $\bet|_{F_k} = \bet_{F_k}$ for every $F_k \in \T_Q$, $Q \in
\Q_w$. Clearly, $\bet \in X_{1, \Omega}(w)$.
\end{proof}

\begin{thm}
\label{solutionform}
 Let $u_0 \in PCR(\Omega)$ and denote $G =
G_{u_0}$. Let $u$ be the global strong solution to \eqref{1-tvflow}. Then there
exist a finite sequence of time instances $0=t_0<t_1<\ldots<
t_n$, partitions  $\Q_0, \ldots, \Q_{n-1}$ of $\Omega$ and
consistent signatures $\S_k$ for $\Q_k$ subordinate to $G$, $k=0,
\ldots, n-1$, such that
\begin{equation}
\label{ut}
u_t = - \frac{1}{\Lb^2(F)} \left( \Hd^1(\partial F^+) - \Hd^1(\partial
F^-)\right) \quad \text{in } ]t_k, t_{k+1}[\times  F \quad \text{for
each } F \in \Q_k,
\end{equation}
$k = 0, 1, \ldots, n-1$ and $u(t, \cdot) = \frac{1}{\Lb^2(\Omega)} \int_\Omega
u_0 \dd \Lb^2$ for $t \geq t_n$. In particular, $u(t, \cdot) \in PCR(\Omega)$
and $\Q_{u(t, \cdot)}$ is subordinate to $G$ for all $t>0$. Furthermore,
\begin{equation}
\label{asymp}
t_n \leq \frac{1}{2}\left(\min_{F 
\subsetneq \Omega, F \in \F(G)} \frac{\Per_1(F, 
\Omega)}{\Lb^2(F)}\right)^{-1} \cdot \max \left|u_0 
- \frac{1}{\Lb^2(\Omega)} \int_\Omega u_0 \dd \Lb^2\right|.
\end{equation}
\end{thm}
\begin{remark}
 Theorem \ref{solutionform} implies that $t \mapsto u(t, \cdot)$ has a 
representative that is Lipschitz with values in $BV(\Omega)$.
\end{remark}

\begin{proof}
 We proceed inductively, starting with $j=0$. Suppose we have proved that there
exist time instances $0=t_0<\ldots< t_j$, partitions  $\Q_k$
of $\Omega$ and consistent signatures $\S_k$ for $\Q_k$ subordinate to $G$,
$k=0, \ldots, j-1$, such that \eqref{ut} holds for $k=1, \ldots, j$ (for $j=0$
this assumption is vacuously satisfied). This implies that $u(t, \cdot) \in
PCR(\Omega)$ and $\Q_{u(t, \cdot)}$ is subordinate to $G$ for $t \in [0, t_j]$.
Let $\Q_j = \T$, $\S_j = \S$, $\z_j = \bet$, where $\T$, $\S$ and $\bet$ are
the partition of $\Omega$, the consistent signature for $\T$ subordinate to $G$
and the vector field produced by Theorem \ref{z} given $w = u(t_j, \cdot)$. For
$T>0$ let us define a function $\widetilde u_j \in C([t_j, T[, BV(\Omega))$ by
 \[\widetilde u_{j,t}(t, \cdot) = \div \z_j \text{ for } t \in ]t_j, T[,\]
 \[\widetilde u_j(t_j, \cdot) = u(t_j, \cdot). \]
Clearly, pair $(\widetilde u_j, \z_j)$ satisfies regularity conditions as well
as (\ref{zeqn}, \ref{zselect}) from Definition \ref{defatvf}. Let us choose
$T=t_{j+1}$ as the first time instance $t>t_j$ such that
\[\widetilde u_j(t, \cdot)|_{F} = \widetilde u_j(t, \cdot)|_{ 
F'},\]
where $F, F' \in \Q_j$, $F \neq F'$, $\Hd^1(\partial F \cap \partial F') >0$;
i.\,e.\;the first moment of merging of facets after time $t_j$.

Due to condition (4) of Theorem \ref{z}, one can show, similarly as in the proof 
of Theorem \ref{minimization}, that condition \eqref{zint} of Definition
\ref{defatvf} is satisfied for $(\widetilde u_j, \z_j)$ and $\widetilde u_j$ is
the solution to \eqref{atv} with initial datum $u(t_j, \cdot)$ in $[t_j,
t_{j+1}[$. Then, due to continuity, $u$ is necessarily equal to $\widetilde u_j$
in $[t_j, t_{j+1}]$, in particular, for $t \in [t_j, t_{j+1}]$, $u(t, \cdot) \in
PCR(\Omega)$ and $\Q_{u(t, \cdot)}$ is subordinate to $G$. This completes the
proof of the induction step.

Now, let us prove that this procedure terminates after a finite number of steps.
For this purpose, we rely on Theorem \ref{divz}. 
In fact, we  prove
that there exists a constant $\gamma=\gamma(G)>0$ such that at each
$t_j$, $j>0$ the non-increasing function $t \mapsto \|\div \z(t,
\cdot)\|_{L^2(\Omega)}$ has a jump of size at least $\gamma$. Here $\z \in 
L^\infty(]0,\infty[\times \Omega)$ satisfies the conditions in Definition 
\ref{defatvf} with $u$ being the strong solution to \eqref{1-tvflow} starting 
with $u_0$. 

First, we argue that $\|\div \z(t_j, \cdot)\|_{L^2(\Omega)} < \|\div \z(t,
\cdot)\|_{L^2(\Omega)}$ for each $t \in [t_{j-1}, t_j[$ (in this interval
$t \mapsto \div
\z(t, \cdot)$ is a constant function). We will reason by contradiction.
If $\|\div
\z(t_j, \cdot)\|_{L^2(\Omega)} = \|\div \z(t_{j-1}, \cdot)\|_{L^2(\Omega)}$,
then $\z(t_{j-1}, \cdot)$ is a minimizer of $\F_\Omega$  in $X_{1,
\Omega}(u(t_j, \cdot))$ and consequently $\div \z(t_{j-1}, \cdot) = \div \z(t_j,
\cdot) = \div \z(t, \cdot)$ for $t \in [t_{j-1}, t_{j+1}[$ (see Theorem
\ref{divz}). According to Lemma \ref{charac_subdiff}, the minimization problem
for $\F_\Omega$ in $X_{1,
\Omega}(u(t_j, \cdot))$ is equivalent to minimization of functionals $\F_Q$
defined by $\F_Q(\bet) = \int_Q \left(\div \bet \right)^2 \dd \Lb^2$
 on the set of vector fields $\bet \in L^\infty(Q, \RR)$ satisfying
 \[\div \bet \in L^2(Q),\qquad |\bet|_\infty \leq 1
\qquad \left. [\bet, \norm^{Q}] \right |_{\partial
Q^\pm} = \mp 1, \qquad\left.[\bet,  \norm^Q]\right |_{\partial Q \cap \partial
\Omega} = 0\]
separately for each $Q \in \Q_{u(t_j, \cdot)}$, where
$(\partial Q^+, \partial Q^+) = \S_{u(t_j, \cdot)}(Q)$. Let us take $Q \in
\Q_{u(t_j, \cdot)}$ such that there exist $F_1, F_2$ in $\Q_{j-1}$, $F_1 \neq
F_2$, $\Hd^1(\partial F_1 \cap \partial F_2) >0$, with $F_1, F_2 \subset Q$.
Denote by $\Q_{j-1,Q}$ the maximal subset of $\Q_{j-1}$ with the properties
\begin{itemize}
 \item $F_1, F_2$ belong to $\Q_{j-1, Q}$,
 \item if $F$ belongs to $\Q_{j-1}$ then $F \subset Q$,
 \item if $F$ belongs to $\Q_{j-1, Q}$ then there exists $F' \in \Q_{j-1, Q}$,
$F' \neq F$ with $\Hd^1(\partial F \cap \partial F') \neq 0$.
\end{itemize}
Let now $F_0$ be a minimizer of $F \mapsto
u_t|_{]t_{j-1}, t_{j+1}[\times  F} = \div \z
|_{]t_{j-1}, t_{j+1}[\times  F}$ among $F \in \Q_{j-1,Q}$. Then, due
to
\eqref{zint} and the way $Du$ changes after the moment of merging, we
necessarily have
\[\left. [\z, \norm^{F_0}]\right|_{]t_j, t_{j+1}[\times \partial F_0^- \cup 
(\partial F_0^+ \setminus \partial Q)} = +1.\]
Due to the choice of
$F_0$, $\Hd^1(\partial F_0^+ \setminus \partial Q) > 0$, hence
\[\left. \frac{1}{\Lb^2(F_0)}\int_{F_0} \div \z \dd \Lb^2\right|_{[t_j, 
t_{j+1}[} > \left. \frac{1}{\Lb^2(F_0)}\int_{F_0} \div \z \dd 
\Lb^2\right|_{[t_{j-1}, t_j[},\]
a desired contradiction.

Next, we observe that there is only a finite set of values, depending only on 
$G$, that $\|\div \z(t, \cdot)\|_{L^2(\Omega)}$ can
achieve. Indeed, for all $t\geq 0$, $\div \z(t, \cdot)$ is the unique result of
minimization problems for $\F_Q$ with $Q \in
\Q_{u(t_j, \cdot)}$, $(\partial Q^+, \partial Q^-) = \S_{u(t_j, \cdot)}(Q)$,
$j=0, 1, \ldots$. Each $\Q_{u(t_j, \cdot)}$ is a partition of $\Omega$
subordinate to $G$, each $\S_{u(t_j, \cdot)}$ is a consistent signature for
$\Q_{u(t_j, \cdot)}$ subordinate to $G$. There is only a finite number of
these.

It remains to prove the estimate on $t_n$. In any time instance $t \geq 0$ the
maximum (minimum) value of $u(t, \cdot)$ is attained in a rectilinear polygon
$F_+(t)$ ($F_-(t)$). In all but a finite number of $t$ we have
\[\mp u_t(t, \cdot) |_{F_\pm(t)} = \mp \div \z(t, \cdot) |_{F_\pm(t)} = 
\frac{{\rm Per}_1(F_\pm(t), \Omega)}{\Lb^2(F_\pm(t))}\geq \min_{F 
\subsetneq \Omega, F \in \F(G)} \frac{\Per_1(F, 
\Omega)}{\Lb^2(F)},\]
unless $F^+ = F^- = \Omega$. Furthermore, testing \eqref{atv} with
$ \chi_\Omega$ yields $\frac{\dd}{\dd t}\int_\Omega u(t, \cdot) \dd \Lb^2 = 0$
in a.\,e.\;$t\geq 0$ and, due to continuity of the semigroup in $L^2$,
\[ \frac{1}{|\Omega|}\int_\Omega u(t, \cdot) \dd \Lb^2 = 
\frac{1}{|\Omega|}\int_\Omega u_0 \dd \Lb^2\]
in all $t> 0$. This concludes the proof.
\end{proof}

\section{The case $\Omega = \RR$}\label{cauchy}
In this section we transfer previous results to the case
$\Omega = \RR$. First, we note that all the definitions and theorems in
subsections \ref{sec:atv} and \ref{sec:atvf} carry over without change (the
Neumann boundary condition becomes void) to this case (see \cite{moll}). As for 
the definitions in subsections \ref{sec:rect}, it turns out that the statements 
of our results transfer nicely to the case of the whole plane if we allow for 
certain \emph{unbounded rectilinear polygons}. Accordingly, in this section a 
subset $F \subset \RR$ will be called a rectilinear polygon if either
\begin{itemize}
 \item $F=\bigcup \R_F$ with a finite $\R_F \subset \R$ (in which case we say 
that $F$ is a bounded rectilinear polygon)
 \item or $F=\overline{\RR \setminus \bigcup \R_F}$ with a finite $\R_F \subset 
\R$ (in which case we say that $F$ is an unbounded rectilinear polygon).
\end{itemize}
Next, we  restrict ourselves to non-negative compactly supported initial data. 
We say that a non-negative compactly supported function $w \in BV(\RR)$ belongs 
to $PCR_+(\RR)$ if there exists a partition $\Q$ of $\RR$ such that $w$ is 
constant in the interior of each $Q \in \Q$. Note that any such $\Q$ contains 
exactly one unbounded set $Q_0$ and $w|_{Q_0} = 0$.

The essential difficulty in obtaining results analogous to Theorems 
\ref{minimization} and \ref{solutionform} lies in dealing with unbounded sets 
that one expects to be produced by a suitable version of the algorithm in 
Section \ref{mini}. For this purpose, we need the following

\begin{lemma}\label{con.calib}
  Let $f\in PCR_+(\RR)$ and let $F$ be an unbounded rectilinear polygon. Then, 
there exists a vector field $\bxi_F \in X_F$ such that
\begin{equation}
\label{xi.calib}
|\bxi_F|_\infty \leq 1, \qquad \div \bxi_F + f|_F = \text{const.}, \qquad 
\quad [\bxi_F, \norm^F] = 1 
\end{equation}
if and only if
\begin{equation}
    \label{cond.cali}
    \Hd^1(\partial^*E \cap \partial F) + \int_E f \dd \Lb^2 \leq {\rm Per}_1(E, 
\interior F)
\end{equation}
for all $E\subset F$ bounded of finite perimeter. Moreover, in this case
$\div \bxi_F + f = 0$ in $F$ for any vector field $\bxi_F$ satisfying
\eqref{xi.calib}.
\end{lemma}
This is a version of \cite[Theorem 5 and Lemma 6]{thetvf} where analogous 
statement is proved for isotropic perimeter in case $f=0$. The idea of the proof 
is to consider auxiliary problem in a large enough ball. The proof of Lemma 
\ref{con.calib} follows along similar lines, however we decided to put it here, 
also because it seems that there is a small gap in the proof of \cite[Theorem 
5]{thetvf} that we patch. Namely the first inequality in line 12, page 511 of 
\cite{thetvf} (corresponding to \eqref{distineq} here) does not seem to be 
satisfied in general.

\begin{proof}
 It is easy to see that if $\bxi_F$ satisfies \eqref{xi.calib} then $\div\bxi_F 
+ f = 0$ in $F$ (see \cite[Lemma 6]{thetvf}). Thus, if a vector field $\bxi_F 
\in X_F$ satisfies \eqref{xi.calib}, then we have for any bounded set $E \subset 
\RR$ of finite perimeter
 \[ 0 = \int_E \div \bxi_F + f \dd \Lb^2 \geq \Hd^1(\partial^*E \cap \partial F) 
+ \int_E f \dd \Lb^2 - {\rm Per}_1(E, \interior F). \]
 Now assume that \eqref{cond.cali} holds. Let us take $R>0$ large enough that
 \begin{equation}
  \label{assR}
  2\, \dist(\partial B_\infty(R), \partial F \cup \supp f) \geq \Hd^1(\partial 
F) + \int_F f \dd \Lb^2.
 \end{equation}
 Put
 $c(R) = - \frac{\Hd^1(\partial F) + \int_F f \dd \Lb^2}{\Hd^1(\partial 
B_\infty(R))}$. Denote by $\bxi_R$ the
minimizer of functional $\F$ defined by
$\F(\bet) = \int_{F \cap B_\infty(R)} \left(\div \bet + f\right)^2 \dd \Lb^2$
on the set of vector fields $\bet \in X_{F \cap B_\infty(R)}$ satisfying
 \[|\bet|_\infty \leq 1, \qquad [\bet,
\norm^F] = 1, \qquad [\bet,
\norm^{B_\infty(R)}] = c(R).\]
If $\div \bxi_R + f$ is
constant in $F \cap B_\infty(R)$ then, due to choice of $c(R)$, $\div \bxi_R + 
f\equiv
0$ in $F \cap B_\infty(R)$. Supposing that the opposite is true, we obtain, as
in the proof of Theorem \ref{minimization}, that there exists $\lambda > 0$
such that
\[ Q_\lambda = \{ \bx \in F \cap B_\infty(R) \colon \div \bxi_R + f> \lambda\}\]
is a set of positive measure and finite perimeter, and we have
\begin{multline}
\label{Qlambdacauchy}
- {\rm Per}_1(Q_\lambda, \interior F \cap B_\infty(R)) + 
\Hd^1(\partial^*Q_\lambda
\cap \partial F) + \int_{Q_\lambda} f \dd \Lb^2 +  c(R) 
\Hd^1(\partial^*Q_\lambda \cap \partial B_\infty(R)) \\ \geq \lambda 
\Lb^2(Q_\lambda) > 0
\end{multline}
which can be rewritten as
\begin{equation}
 \label{Qlambdacauchy2}
 -{\rm Per}_1(Q_\lambda) + 2 \Hd^1(\partial^*Q_\lambda
\cap \partial F) + \int_{Q_\lambda} f \dd \Lb^2 + (1+ c(R)) 
\Hd^1(\partial^*Q_\lambda \cap \partial B_\infty(R)) > 0.
\end{equation}
Assumption \eqref{assR} implies that $c(R)>-1$, so we  approximate $Q_\lambda$ 
with a closed smooth set as in the proof of Lemma \ref{rectapprox} in such a way 
that \eqref{Qlambdacauchy2} still holds. Due to additivity of left hand side of 
\eqref{Qlambdacauchy2}, there is a connected component $\widetilde Q_\lambda$ of 
this smooth set that also satisfies \eqref{Qlambdacauchy2}, or equivalently
\begin{equation}
 \label{wtQlambda}
 - {\rm Per}_1(\widetilde Q_\lambda, \interior F \cap B_\infty(R)) + 
\Hd^1(\partial \widetilde Q_\lambda
\cap \partial F) + \int_{\widetilde Q_\lambda} f \dd \Lb^2 +  c(R) 
\Hd^1(\partial \widetilde Q_\lambda \cap \partial B_\infty(R)) > 0.
\end{equation}

If $\partial \widetilde Q_\lambda \cap \partial B_\infty(R) = \emptyset$,
\eqref{wtQlambda} contradicts \eqref{cond.cali}. On the other hand, if $\partial
\widetilde Q_\lambda \cap (\partial F \cup \supp f) = \emptyset$,
\eqref{wtQlambda} itself is a contradiction (recall that $c(R) \leq 0$).
Taking these observations into
account, there necessarily holds
\begin{equation}
\label{distineq}
{\rm Per}_1(\widetilde Q_\lambda, \interior F \cap B_\infty(R)) \geq 2 \,
\dist(\partial B_\infty(R), \partial F \cup \supp f),
\end{equation}
whence \eqref{assR} yields a contradiction, \emph{unless} $\widetilde Q_\lambda$
is not simply connected in such a way that there is a connected component
$\Gamma$ of $\partial \widetilde Q_\lambda$ such that
\begin{itemize}
 \item $\widetilde Q_\lambda$ is inside of $\Gamma$ ($\Gamma$ is the exterior
boundary of $\widetilde Q_\lambda$),
 \item $\Gamma$ does not intersect $\partial F \cup \supp f$,
 \item and $\Gamma$ intersects all four sides of $B_\infty(R)$.
\end{itemize}
In this case, let us denote by $\widehat Q_\lambda$ the union of $\widetilde
Q_\lambda$ and the region between $\Gamma$ and $\partial B_\infty(R)$. We have
$\int_{\Gamma \setminus \partial B_\infty(R)} |\norm^{\widetilde Q_\lambda}|_1
\dd \Hd^1 \geq \Hd^1(B_\infty(R) \setminus \Gamma)$ and consequently (as $-
c(R)<1$)
\begin{multline*} {\rm Per}_1(\widetilde Q_\lambda, \interior F \cap 
B_\infty(R)) - c(R) \Hd^1(\partial \widetilde Q_\lambda \cap \partial 
B_\infty(R)) \geq {\rm Per}_1(\widehat Q_\lambda, \interior F \cap B_\infty(R)) 
\\ - c(R) \Hd^1(\partial \widehat Q_\lambda \cap \partial B_\infty(R))  = {\rm 
Per}_1(\widehat Q_\lambda, \interior F \cap B_\infty(R)) + \Hd^1(\partial F) + 
\int_F f \dd \Lb^2,
\end{multline*}
a contradiction with \eqref{wtQlambda} which implies that $\div \bxi_R \equiv 
0$.
Now, we define $\bxi_F \in X_F$ by
\begin{equation}
 \bxi_F(x_1, x_2) = \left\{\begin{array}{ll}
                  \bxi_R(x_1, x_2) & \text{in } F \cap B_\infty(R), \\
                  (c(R) \sgn x_1, 0) & \text{in } \{|x_1| > R, |x_2| < R\} \\
                  (0, c(R) \sgn x_2) & \text{in } \{|x_1| < R, |x_2| > R\} \\
                  (0, 0) & \text{in } \{|x_1| > R, |x_2| > R\}. \\
                 \end{array}\right.
\end{equation}
\end{proof}

Now given $\lambda >0$, $u_0 \in PCR_+(\RR)$, grid $G=G_f$, and an unbounded
rectilinear polygon $F$ subordinate to $G$ with signature $(\partial F^+,
\partial F^-) = (\emptyset, \partial F)$, let us denote by $R_0$ the smallest
rectangle containing the support of $u_0$ (clearly $R_0$ is subordinate to $G$
and $\partial F\subset R_0$). Next, suppose that there is a set of finite
perimeter $E \subset F$ such that $\J_{F, \partial F^+, \partial F^-,
\frac{u_0}{\lambda}}(E) < 0$. Then there holds $\J_{F, \partial F^+, \partial
F^-, \frac{u_0}{\lambda}}(R_0 \cap E) \leq \J_{F, \partial F^+, \partial F^-,
\frac{u_0}{\lambda}}(E)$. Indeed, we only need to argue that ${\rm Per}_1(R_0
\cap E) \leq {\rm Per}_1(E)$, which follows easily by approximation of $E$ with 
smooth
sets. This way, we obtain the following alternative:
\begin{itemize}
 \item either $\J_{F, \partial F^+, \partial F^-, \frac{u_0}{\lambda}}$ is 
minimized by a bounded rectilinear polygon subordinate to $G$,
 \item or ${\rm Per}_1(E, \interior F) - \Hd^1(\partial^*E \cap \partial F) - 
\int_E \frac{u_0}{\lambda} \dd \Lb^2 \geq 0$ for each $E \subset F$ of finite 
perimeter.
\end{itemize}
By virtue of Lemma \ref{con.calib}, in the second case there exists a vector
field $\bxi_F \in X_F$ such that \eqref{xi.calib} is satisfied. Supplementing
the proof of Theorem \ref{minimization} with this reasoning we obtain that it
holds for $\Omega = \RR$, provided that $u_0 \in PCR_+(\RR)$. By a similar
modification in section \ref{facets}, Theorem \ref{solutionform} also holds
for $\Omega = \RR$ with the same provision on $u_0$. In place of \eqref{asymp}
we get the following estimate on the extinction time $t_n$ after which $u=0$:
\begin{equation}
t_n \leq \frac{\Lb^2(R_0)}{{\rm Per}_1(R_0)} \cdot \ess \max u_0,
\end{equation}
where we denoted by $R_0$ the minimal rectangle containing the support of
$u_0$. Its particulalry simple form as compared to \eqref{asymp} is due to the
fact that $R_0$ clearly minimizes $\frac{{\rm Per}_1(F, \RR)}{\Lb^2(F)}$ among 
bounded
rectilinear polygons subordinate to $G$.

\section{Preservation of continuity in rectangles}\label{regularity}
We start with a lemma concerning $PCR$ functions on a rectangle, which says, 
roughly speaking, that the maximal oscillation on horizontal (or vertical) 
lines, on any given length scale, is not increased by the solution to 
\eqref{1-min} with respect to initial datum $u_0 \in PCR(\Omega)$. To make a 
precise statement, we fix a rectangle $\Omega$ and let $G$ be any grid such 
that $\Omega$ is subordinate to $G$. Further, let $m_1$ ($m_2$) be the number 
of horizontal (vertical) lines of $G$. For any given integer $0\leq m \leq 
m_1-3$ ($m_2-3$) we denote by $\R^2_{1,m}(G)$ ($\R^2_{2,m}(G)$) the set of 
pairs 
of rectangles that lay in
the strip of $\Omega$ between any two successive horizontal (vertical) lines of 
$G$ and are separated by at most $m$ rectangles in
$\R(G)$.

\begin{lemma}
\label{lemmacont}
Let $\Omega$ be a rectangle and let $u$ be the solution to \eqref{1-min} with 
$u_0 
\in PCR(\Omega)$, $\lambda>0$. Let $G$ be a grid such that 
$\Q_{u_0}$ is subordinate to $G$. For $i=1,2$ there holds
\begin{equation}
\label{jumpineq}
 \max_{(R_1, R_2) \in \R^2_{i,m}(G)} \left|u|_{R_1} - u|_{R_2}\right| \leq 
\max_{(R_1, R_2) \in \R^2_{i,m}(G)} \left|u_0|_{R_1} -
u_0|_{R_2}\right|.
\end{equation}
\end{lemma}

\begin{remark}
Taking $m=0$ in Lemma \ref{lemmacont} we obtain that
\[\Hd^1\!\!-\!\ess\max_{J_u} (u_+ - u_-) \leq \Hd^1\!\!-\!\ess\max_{J_{u_0}} 
(u_{0,+} - u_{0,-}).\]
\end{remark}

\begin{proof}
For a given $0 \leq k \leq m_i$ assume we have already proved that 
\eqref{jumpineq} holds for each $0 \leq m < k$. Take any pair of rectangles 
$(R_+, R_-) \in \R^2_{i,k}(G)$
that realizes the maximum in $\left|u|_{R_1} - u|_{R_2}\right|$. Let us take 
rectilinear polygons $F_+, F_-$ in $Q_u$ such that $R_\pm \subset 
F_\pm$.

Now, we  assume that $i=1$ (i.\,e.\;rectangles $R_\pm$
are in the same row of $\R(G)$), $u|_{R_+} > u|_{R_-}$ and $R_-$ is to the left 
of $R_+$. Let us denote by $x_-$ the
maximal value of $x$ coordinate of points in $R_-$ and by $x_+$ the minimal
value of $x$ coordinate of points in $R_+$. Further, let us denote 

\begin{figure}[h]
    \begin{center}
    \includegraphics[width=0.4\textwidth]{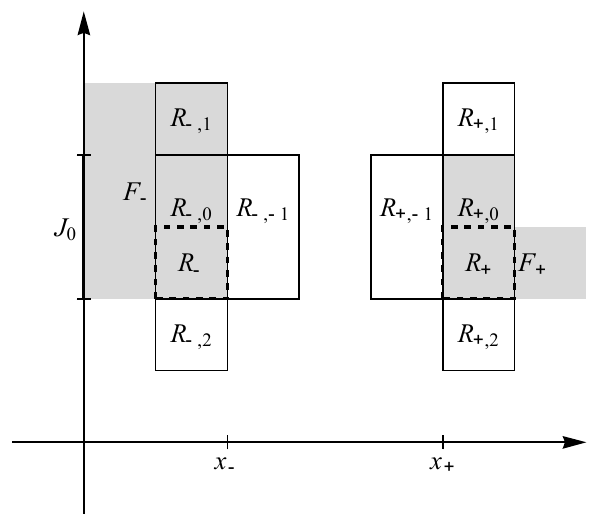} 
    \end{center}
  \caption{Map key for the notation in the proof of Lemma \ref{lemmacont}.}
\end{figure}

\begin{itemize}
 \item by $J_0$ the maximal interval such that $\{x_\pm\}\times J_0 \cap
\partial R_\pm \neq \emptyset$ and $\{x_\pm\}\times J_0 \subset \partial
F_\pm$,
\item by $R_{\pm,0}$ minimal \emph{rectangles} in $\F(G)$ that have 
$\{x_\pm\}\times
J_0$ as one of their sides and contained $R_\pm$,
\item by $R_{+,-1}$ (resp. $R_{-,-1}$) the minimal rectangle in $\F(G)$ that has 
$\{x_+\}\times
J_0$ (resp. $\{x_-\}\times J_0$) as one of its sides and does not contain $R_+$ 
(resp. $R_-$), 
\item by $K$ the number of endpoints of $J_0$ that do not intersect $\partial 
\Omega$ ($K \in \{0,1,2\})$, 
\item by $R_{\pm,j}$, $j \in \mathbb N$, $j \leq K$ the $K$ pairs of rectangles 
in $\R(G)$ such
that
\begin{itemize}
\item all of $R_{+,j}$, have a common side with
$R_{+,0}$ and belong to the same column in $\R(G)$ as $R_+$,
\item all of $R_{-,j}$, have a common side with $R_{-,0}$ and belong to
the same column in $\R(G)$ as $R_-$,
\item for a fixed $j$, both $R_{\pm, j}$ belong to the same row in $\R(G)$.
\end{itemize}
\end{itemize}
Due to the way these are defined, fixing $j\leq K$, at least one of the two
rectangles $R_{\pm, j}$ is not contained in $F_+ \cup F_-$, and  
\begin{equation}
\label{rectineq1}
\text{at least one of inequalities } u|_{R_{+, j}}\!<\!u|_{R_{+,0}}, \; 
u|_{R_{-, j}}\!>\!u|_{R_{-,0}} \text{ hold.}
\end{equation} 

If there is a pair of rectangles $R'_\pm$ in $\R^2_{1, m}(G)$, $m < k$ such
that $R'_\pm \subset F_\pm$ then we have already proved
that \[\left|u|_{R_+} - u|_{R_-}\right| = \left|u|_{R'_+} - u|_{R'_-}\right| 
\leq \max_{(R_1, R_2) \in \R^2_{i,k}} \left|u_0|_{R_1} -
u_0|_{R_2}\right|.\]
Therefore, we can assume that
\begin{equation}
\label{rectineq2}
u|_{R_{+, -1}}\!<\!u|_{R_{+,0}} \text{ and }  u|_{R_{-, -1}}\!>\!u|_{R_{-,0}} 
\text{ hold.}
\end{equation} 

We have
\[F_- \in \argmin \J_{F_-, \partial F_-^+, \partial F_-^-, \frac{u_0}{\lambda}} 
\quad \text{and}
\quad F_+ \in
\argmax \check \J_{F_+, \partial F_+^+, \partial F_+^-, \frac{u_0}{\lambda}}\]
(see the Remarks before the statement of Theorem \ref{minimization} and after 
its proof). Therefore, taking into account (\ref{rectineq1}, \ref{rectineq2}),
\begin{multline}
\label{rectcomp}
u|_{F_+} - u|_{F_-} = - \lambda\left(\check \J_{F_+, \partial
F_+^+, \partial F_+^-, \frac{u_0}{\lambda}}(F_+) - \J_{F_-, \partial F_-^+, 
\partial F_-^-, \frac{u_0}{\lambda}}(F_-)\right) \\
\leq - \lambda\left(\check \J_{F_+, \partial F_+^+, \partial F_+^-, 
\frac{u_0}{\lambda}}(R_+) + \J_{F_-, \partial F_-^+, \partial F_-^-, 
\frac{u_0}{\lambda}}(R_-)\right) \\ \leq \frac{1}{\Lb^2(R_+)} \int_{R_+} u_0 
\dd \Lb^2 - \frac{1}{\Lb^2(R_-)} \int_{R_-} u_0 \dd \Lb^2 \leq \max_{(R_1, R_2) 
\in \R^2_{i,k}} \left|u_0|_{R_1} - u_0|_{R_2}\right|.
\end{multline}
\end{proof}

\begin{thm}
 \label{contpres}
Let $\Omega$ be a rectangle and let $u$ be the solution to \eqref{1-min} with 
$u_0 
\in C(\Omega)$, $\lambda>0$. Then $u \in C(\Omega)$. In fact, if $\omega_1, 
\omega_2 \colon [0, \infty[ \to [0,
\infty[$ are continuous functions such that
\[|u_0(x_1, x_2) - u_0(y_1, y_2)| \leq \omega_1(|x_1-y_1|) + \omega_2(|x_2 -
y_2|)\]
for each $(x_1, x_2)$, $(y_1, y_2)$ in $\Omega$ then we have
\[|u(x_1, x_2) - u(y_1, y_2)| \leq \omega_1(|x_1-y_1|) + \omega_2(|x_2
- y_2|)\]
for each $t>0$,$(x_1, x_2)$, $(y_1, y_2)$ in $\Omega$.
\end{thm}

\begin{remark}
Note that if $\omega$ is a concave modulus of continuity for $u_0$ with respect
to norm $|\cdot|_1$, then $\omega_1, \omega_2$ defined by
$\omega_1=\omega_2=\omega$ satisfy the assumptions of Theorem \ref{contpres}.
On the other hand, given $\omega_1, \omega_2$ as in the Theorem, $\omega' =
\omega_1 + \omega_2$ is a modulus of continuity for $u_0$ (as well as 
$u$). Theorem \ref{contpres} implies for instance that if $L$ is the Lipschitz
constant for $u_0$ with respect to norm $|\cdot|_1$, then the Lipschitz
constant of $u$ with respect to norm $|\cdot|_1$ is not greater than
$L$.
\end{remark}

\begin{proof}
We denote $\Omega = (s_1,s_2) + [0,l_1]\times[0, l_2]$. For $k=1,2,\ldots$ let
\[G_k = \left(\left\{s_1 +
\frac{j l_1}{k}, j = 0,1,\ldots,k\right\} \times \mathbb R\right) \cup
\left(\mathbb R \times \left\{s_2 + \frac{j l_2}{k}, j =
0,1,\ldots,k\right\}\right)\]
and let $u_{0, k} \in PCR(\Omega)$ be
defined by
\[u_{0, k}(\bx) = u_0(\bx_R) \quad \text{for } \bx \in R \in
\R(G_k),\]
where $\bx_R$ is the center of $R$.

For any $k=1,2,\ldots$, $i=1,2$, $m=0,\ldots,k-1$, let $(R,R') \in
\R^2_{i,m}(G_k)$ with $(x_1, x_2) \in R, (y_1, y_2) \in R'$
we have
\begin{equation}
\label{initialcont}
 \left|u_{0, k}|_{R} - u_{0, k}|_{R'}\right| \leq \omega_i(|x_i -
y_i| + \tfrac{1}{k}).
\end{equation}

Let us denote by $u_k$ the solution to \eqref{1-min} with datum $u_{0, k}$. Due 
to inequality \eqref{initialcont} and Lemma
\ref{lemmacont} we have for any $(x_1, x_2), (y_1, y_2) \in \Omega$,
\begin{multline*}
|u_k(x_1,x_2) - u_k(y_1, y_2)| \leq |u_k(x_1,x_2)
- u_k(y_1, x_2)| + |u_k(y_1,x_2) - u_k(y_1, y_2)| \\ \leq
\omega_1(|x_1 - y_1| + \tfrac{1}{k}) + \omega_2(|x_2 - y_2| + \tfrac{1}{k}).
\end{multline*}

Now, due to monotonicity of $- \partial TV_{1,\Omega}$, we have
$\|u_k - u\|_{L^2(\Omega)}\leq \|u_{0, k} - u_0\|_{L^2(\Omega)}.$
Therefore, there exists a set $N \subset \Omega$ of zero $\Lb^2$ measure and a 
subsequence $(k_j)$ such that
$ u_{k_j}(\bx) \to u(\bx) \quad \text{for all } \bx \in \Omega\setminus N.$
Now, for each pair $(x_1, x_2), (y_1, y_2) \in \Omega$ take any pair of
sequences $((x_{1,n}, x_{2,n})), ((y_{1,n}, y_{2,n})) \subset \Omega\setminus N$ 
such that $x_{i,n} \to x_i$ and $y_{i,n} \to
y_i$. Passing to the limit $j \to \infty$ and then $n \to \infty$ in
\[ |u_{k_j}(x_{1,n}, x_{2,n}) - u_{k_j}(y_{1,n}, y_{2,n})| \leq
\omega_1(|x_{1,n} - y_{1,n}| + \tfrac{1}{k_j}) + \omega_2(|x_{2,n} - y_{2,n}| + 
\tfrac{1}{k_j})\]
we conclude the proof.
\end{proof}

Analogous results can be obtained for the solutions to \eqref{1-tvflow}. 

\begin{lemma}
\label{lemmacontf}
Let $\Omega$ be a rectangle, let $u$ be the solution to \eqref{1-tvflow} 
with $u_0 \in PCR(\Omega)$ and let $G$ be a grid such that $\Q_{u_0}$ is 
subordinate to $G$. Then for $i=1,2$ there holds
\begin{equation}
\label{jumpineqf}
 \max_{(R_1, R_2) \in \R^2_{i,m}(G)} \left|u(t, \cdot)|_{R_1} - u(t,
\cdot)|_{R_2}\right| \leq \max_{(R_1, R_2) \in \R^2_{i,m}} \left|u_0|_{R_1} -
u_0|_{R_2}\right|
\end{equation}
in any time instance $t>0$.
\end{lemma}

\begin{proof}
The form of solution obtained in Theorem \ref{solutionform} implies that the
function
\[t \mapsto \max_{(R_1, R_2) \in \R^2_{i,m}(G)} \left|u(t, \cdot)|_{R_1} -
u(t, \cdot)|_{R_2}\right|\]
is piecewise linear and continuous, in particular
it does not have jumps. Having this observation in mind, let us consider time 
instance
$\tau > 0$ that does not belong to the set of merging times $\{t_1, \ldots,
t_n\}$.

For a given $0 \leq k \leq m_i$ assume we have already proved that
the slope of
\[t \mapsto \max_{(R_1, R_2) \in \R^2_{i,m}(G)} \left|u(t,
\cdot)|_{R_1} - u(t, \cdot)|_{R_2}\right|\]
is non-positive in $t=\tau$ for
each $0 \leq m < k$. Take any pair of rectangles $(R_+, R_-) \in \R^2_{i,k}(G)$
that realizes the maximum in $\left|u(t, \cdot)|_{R_1} - u(t,
\cdot)|_{R_2}\right|$. Let us take rectilinear polygons $F_+, F_-$ in $\Q_{u(t,
\cdot)}$ such that $R_\pm \subset F_\pm$.

Then, reasoning as in the proof of Lemma \ref{lemmacont}, we obtain
\begin{multline}
\label{rectcompf}
u_t(\tau, \cdot)|_{F_+} - u_t(\tau, \cdot)|_{F_-} = - \check \J_{F_+, \partial
F_+^+,
\partial F_+^-}(F_+) + \J_{F_-, \partial F_-^+, \partial F_-^-}(F_-) \\
\leq - \check \J_{F_+, \partial F_+^+, \partial
F_+^-}(R_+) + \J_{F_-, \partial F_-^+, \partial F_-^-}(R_-) \leq 0
\end{multline}
which concludes the proof. 
\end{proof}

Now, note that if $u$ and $v$ are two solutions to \eqref{atv} with 
$L^2(\Omega)$ initial data $u_0$ and $v_0$ respectively, we have for each time 
instance $t>0$ (see \cite{moll}, Theorem 11.)
\[ \|u(t, \cdot) - v(t, \cdot)\|_{L^2(\Omega)}\leq \|u_0 - v_0\|_{L^2(\Omega)} 
.\]
Using this fact and Lemma \ref{lemmacontf}, we can obtain the following analog 
of Theorem \ref{contpres} for solutions of \eqref{atv}. The proof is almost 
identical and we omit it. 
\begin{thm}
 \label{contpresf}
 Let $\Omega$ be a rectangle and let $u$ be the solution to \eqref{atv} with 
initial datum $u_0 \in C(\Omega)$. Then $u(t, \cdot) \in C(\Omega)$ in
every $t>0$. In fact, if $\omega_1, \omega_2 \colon [0, \infty[ \to [0,
\infty[$ are continuous functions such that
\[|u_0(x_1, x_2) - u_0(y_1, y_2)| \leq \omega_1(|x_1-y_1|) + \omega_2(|x_2 -
y_2|)\]
for each $(x_1, x_2)$, $(y_1, y_2)$ in $\Omega$ then we have
\[|u(t,(x_1, x_2)) - u(t,(y_1, y_2))| \leq \omega_1(|x_1-y_1|) + \omega_2(|x_2
- y_2|)\]
for each $t>0$,$(x_1, x_2)$, $(y_1, y_2)$ in $\Omega$.
\end{thm}

Finally, we note that all the results in this section carry over in a 
straightforward way to the case $\Omega = \RR$, provided that in the statements 
of Theorems \ref{contpres} and \ref{contpresf} $C(\Omega)$ is replaced with 
$C_{c, +}(\RR)$ (meaning non-negative, compactly supported continuous 
functions on $\RR$). On the other hand, if $\Omega$ is a rectilinear polygon 
different from a rectangle, the continuity is not necessarily preserved as 
Example \ref{excrack} shows. 

\section{Examples}\label{examples}
We start with the following general fact showing that minimizing \eqref{AROF}
and solving \eqref{atv} is equivalent in some cases.
\begin{thm}
 \label{equiv}
 Let $\Omega$ be a bounded domain or $\Omega = \RR$ and let $u$ be the strong
solution to \eqref{atv} in $[0,T[$ with initial datum $u_0 \in L^2(\Omega)$. If
there exists $\z \in L^\infty(]0,T[\times\Omega, \RR)$ satisfying conditions
(\ref{zeqn}-\ref{zint}) such that for some $0< \lambda <T$ and almost all
$0<t<\lambda$
 \begin{equation}
  \label{assequiv}
  \int_{\Omega} (\z(t, \cdot), Du(\lambda, \cdot)) = \int_{\Omega}
|Du(\lambda,\cdot)|_\varphi.
 \end{equation}
 Then the minimizer of \eqref{AROF} with $\lambda>0$ is given by
$u_\lambda=u(\lambda, \cdot)$.
\end{thm}

\begin{proof}
Let $\z_\lambda = \frac{1}{\lambda} \int_0^\lambda \z(t, \cdot) \dd t$. Clearly,
$\z_\lambda \in X_\Omega(u_\lambda)$ satisfies $[\z_\lambda, \norm^\Omega] = 
0$ and $|\z_\lambda|^*_\varphi \leq 1$. Furthermore, by virtue of 
\eqref{assequiv},
\begin{equation}
\label{eqanzmeasureelip}
\int_{\Omega} (\z_\lambda,D u_\lambda) = \int_{\Omega} |D u_\lambda|_\varphi,
\end{equation}
so $\z_\lambda \in X_{\varphi, \Omega}(u_\lambda)$. Finally, \eqref{zeqn} 
implies
\begin{equation}
\label{eqdisteliptic}
u_\lambda-u_0=\lambda \div \z_\lambda \quad \text{ in } \mathcal D'(\Omega).
\end{equation}
\end{proof}
One class of solutions to \eqref{atv} such that \eqref{assequiv} holds, are
$PCR$ solutions constructed in Theorem \eqref{solutionform} up to the
first (positive) breaking time, as defined in the following
\begin{defn}
  Let $u \in W^{1,\infty}([0,\infty[, BV(\Omega))$ be the global strong
solution to
\eqref{1-tvflow} with $u_0\in PCR(\Omega)$ and let $0<t_1<\ldots t_n$ be the
sequence of time instances obtained in Theorem
\ref{solutionform}. We call each of $t_1,
\ldots, t_n$ a \emph{merging time}. We say that $t_i$, $i=1,\ldots,n$ is a
\emph{(positive) breaking time} if $\Hd^1(J_{u(t,\cdot)}\setminus J_{u(t_i,
\cdot)}) > 0$ for $t \in ]t_i, t_{i+1}[$.
\end{defn}
Indeed, let $t_k>0$ be the first breaking time and $0 < t < \lambda \leq
t_k$. Then $J_{u(\lambda, \cdot)} \subset J_{u(t, \cdot)}$ up to a
$\Hd^1$-null
set and $\frac{Du(\lambda, \cdot)}{|Du(\lambda, \cdot)|} = \frac{Du(t,
\cdot)}{|Du(t, \cdot)|}$ holds $|Du(\lambda, \cdot)|$-a.\,e.,\;which implies
\eqref{assequiv}.

Now we provide several examples illustrating the strength of our results. Note
that even though they are formulated in the language of the flow, in all of
them $]0, \lambda[ \ni t \mapsto \z(t, \bx)$ is constant
a.\,e.\;for $|Du(\lambda,\cdot)|$-almost every $\bx \in \Omega$ which implies
\eqref{assequiv}, and therefore they are also solutions to \eqref{1-min}.

 Theorem \ref{solutionform} predicts that the jump set of a function piecewise
constant on rectangles may expand under the $TV_1$ flow, i.\,e.\;facet
breaking may occur. Many explicit examples of this kind can be
constructed. Here we present a simple one, for which the procedure described in 
the proof of Theorem \ref{z} is concise enough to be presented in detail.  
\begin{ex}
\label{exbreaking}
Let
\[u(t,\cdot) = \left(1 - \tfrac{4}{3}t\right)_+\chi_B + \left(1 - 2
t\right)_+ \chi_C,\]
where we denoted 
\[B = B_\infty\left((0,0),\tfrac{3}{2}\right), \quad C = 
B_\infty\left((2,0), \tfrac{1}{2}\right) \cup
B_\infty\left((-2,0), \tfrac{1}{2}\right) \cup B_\infty\left((0,2),
\tfrac{1}{2}\right) \cup B_\infty\left((0,-2), \tfrac{1}{2}\right).\] 
For each 
$t \geq 0$, $u(t, \cdot) \in PCR_+(\RR)$ and $u$ solves \eqref{1-tvflow} 
with initial datum $u_0 = \chi_{B \cup C}$. To see this, we 
execute the algorithm described in Theorem \ref{z}. Let $Q_1 = u_0^{-1}(1) = B 
\cup C$. Due to symmetry, the only plausible largest minimizers of 
$\J_{Q_1, \partial Q_1, \emptyset, 0}$ are $B$, $C$ and $B \cup C$ (we only 
need to consider elements of $\F_{u_0}$ and no subset of square $B$ can 
produce lower value of the functional than $B$). We check that the values of 
$\J_{Q_1, \partial Q_1, \emptyset, 0}$ on these sets are, respectively, 
$\frac{4}{3}$, $4$, and $\frac{20}{13}$, hence $B$ 
is the minimizer and the initial velocity on $B$ is $-\frac{4}{3}$. Next, we 
have to find the largest minimizer of $\J_{C, \partial Q_1, \partial B, 0}$. 
There is only one competitor, $C$. To find initial velocity on $C$, we 
calculate 
$- \J_{C, \partial Q_1, \partial B, 0}(C) = -2$. Finally, 
as explained in section \ref{cauchy}, we need to find the largest minimizer 
of $\J_{Q_0, \partial R_0, \partial Q_1, 0}$, where we denoted $R_0$ to be the 
smallest rectangle (square) containing the support of $u_0$ and $Q_0 = R_0 \cap 
u_0^{-1}(0)$. We check that the minimizer is $Q_0$ itself, with $\J_{Q_0, 
\partial R_0, \partial Q_1, 
0}(Q_0) =0$.  

\begin{figure}[h]
    \centering
    \begin{subfigure}{0.3\textwidth}
        \includegraphics[width=\textwidth]{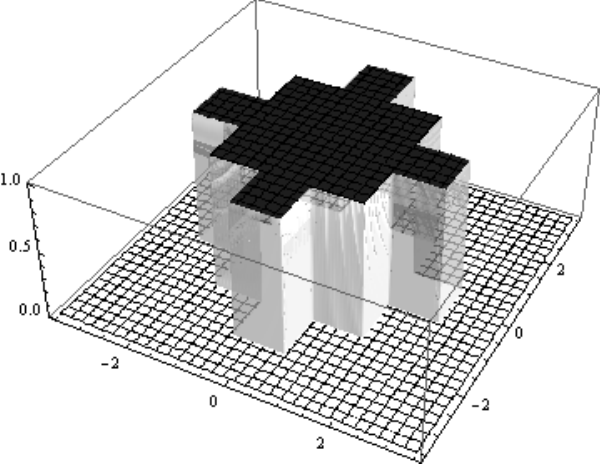}
        \caption{$t=0$.}
        \label{fig:exbreakingt0}
    \end{subfigure}
    \;
    \begin{subfigure}{0.3\textwidth}
        \includegraphics[width=\textwidth]{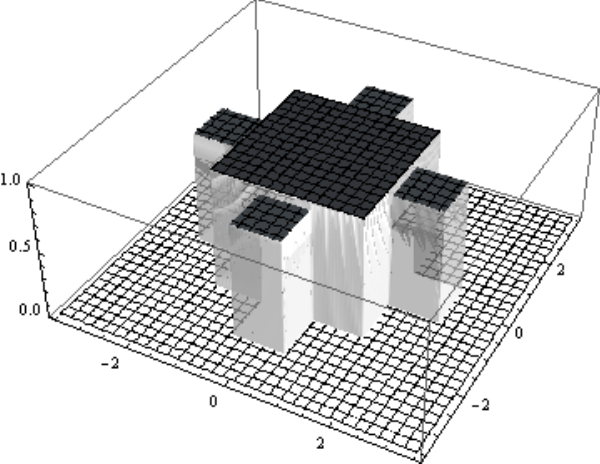}
        \caption{$t=0.08$.}
        \label{fig:exbreakingt0.08}
    \end{subfigure}
    \;
    \begin{subfigure}{0.3\textwidth}
        \includegraphics[width=\textwidth]{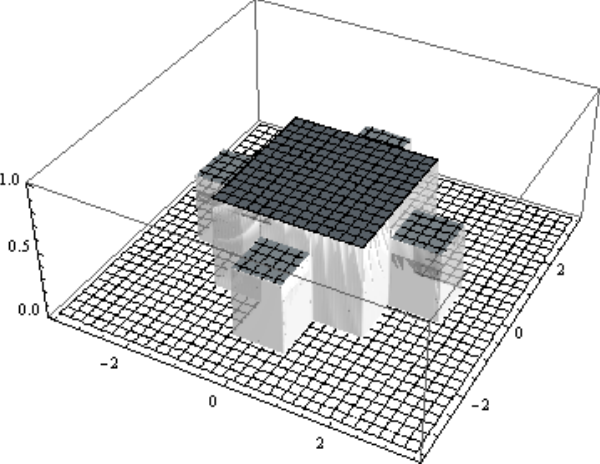}
        \caption{$t=0.24$.}
        \label{fig:exbreakingt0.24}
    \end{subfigure}
    \caption{Plots of $u(t,\cdot)$ from Example \ref{exbreaking} in certain time
instances $t$.}
\label{fig:exbreaking}
\end{figure}
\end{ex}

\begin{figure}[h]
    \centering
    \begin{subfigure}{0.3\textwidth}
        \includegraphics[width=\textwidth]{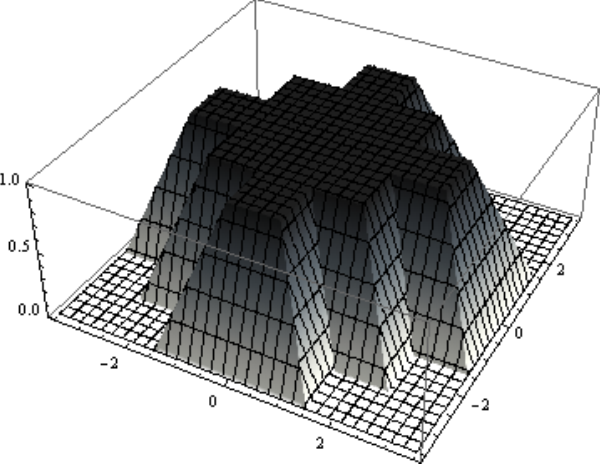}
        \caption{$t=0$.}
        \label{fig:excontt0}
    \end{subfigure}
    \;
    \begin{subfigure}{0.3\textwidth}
        \includegraphics[width=\textwidth]{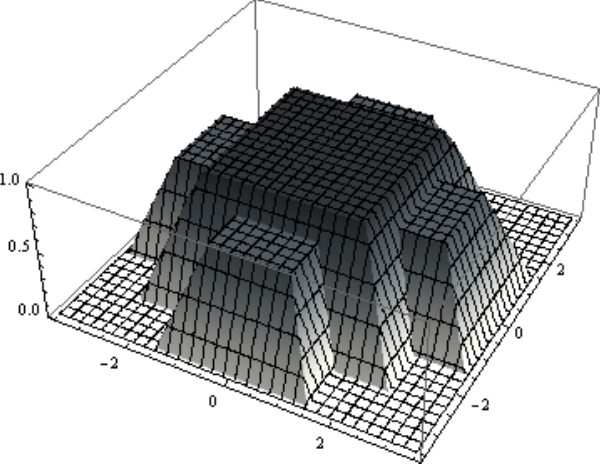}
        \caption{$t=0.08$.}
        \label{fig:excontt0.08}
    \end{subfigure}
    \;
    \begin{subfigure}{0.3\textwidth}
        \includegraphics[width=\textwidth]{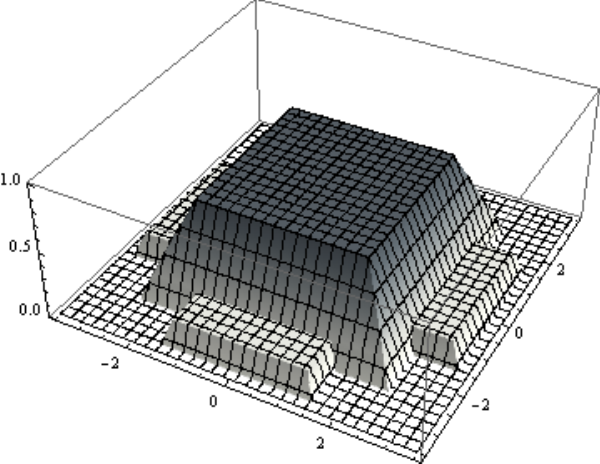}
        \caption{$t=0.24$.}
        \label{fig:excontt0.24}
    \end{subfigure}
    \caption{Plots of $u(t,\cdot)$ from Example \ref{excont} in certain time
instances $t$.}
\label{fig:excont}
\end{figure}
\begin{figure}[h!]
    \centering
    \begin{subfigure}{0.3\textwidth}
        \includegraphics[width=\textwidth]{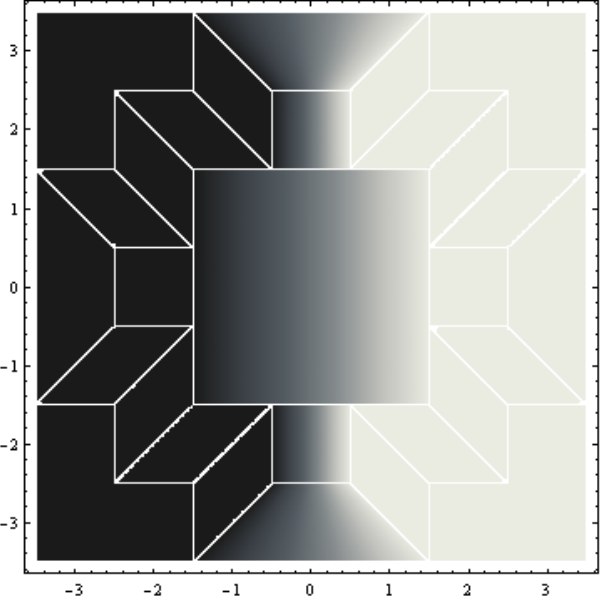}
        \caption{$t=0$.}
        \label{fig:excontzt0}
    \end{subfigure}
    \;
    \begin{subfigure}{0.3\textwidth}
        \includegraphics[width=\textwidth]{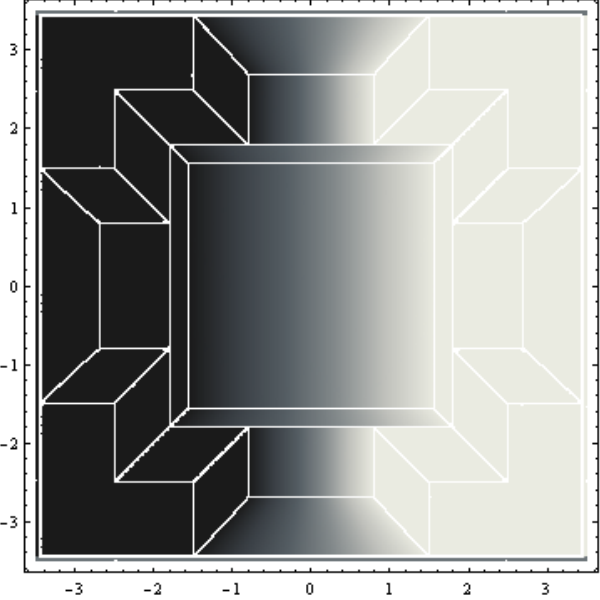}
        \caption{$t=0.08$.}
        \label{fig:excontzt0.08}
    \end{subfigure}
    \;
    \begin{subfigure}{0.3\textwidth}
        \includegraphics[width=\textwidth]{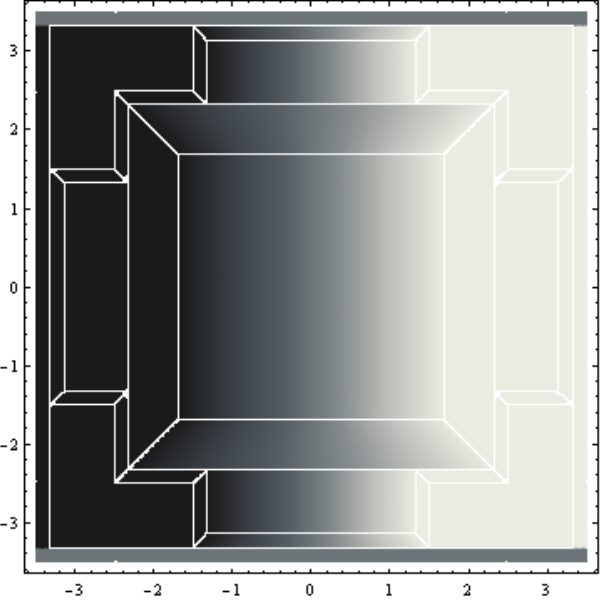}
        \caption{$t=0.24$.}
        \label{fig:excontzt0.24}
    \end{subfigure}
    \caption{Density plots of $z_1(t,\cdot)$ corresponding to $u(t,\cdot)$ from
Example \ref{excont} in certain time instances $t$. Black corresponds to
value $1$, ivory to $-1$; note the value $0$ outside the minimal strip
containing the support of $u(t, \cdot)$.}
\label{fig:excontz}
\end{figure}

On the other hand, Theorem \ref{contpres} asserts that if $u_0$ is (Lipschitz)
continuous,
the solution $u$ starting with $u_0$ is (Lipschitz) continuous in every time 
instance
$t>0$. For instance, if one extends the characteristic function form Example
\ref{exbreaking} continuously outside its support, no jumps will appear in the
evolution --- another manifestation of nonlocality of the equation.
\begin{ex}
\label{excont}
Here we
present Figure \ref{fig:excont}, depicting evolution $u$ of piecewise
linear continuous function $u_0$ obtained by extending the initial datum from
Example \ref{exbreaking} outside its support up to $0$ in such a way that
$\nabla u_0 \in \{(0,0),(1,0), (0,1)\}$. The evolution is
obtained by explicit identification of corresponding field $\z=(z_1,
z_2)$ under an ansatz that in each of a finite number of evolving regions
either $z_i = \pm 1$ or a $z_i$ is a linear interpolation of boundary
values, $i=1,2$ (see Figure \ref{fig:excontz}). This reduces the problem to
a decoupled infinite system of ODEs. The evolution obtained this way is the
strong solution starting with $u_0$ as it satisfies all the requirements in
Definition \ref{defatvf}. Figures \ref{fig:excont} and \ref{fig:excontz} are
obtained by solving numerically the system of ODEs using \emph{Mathematica}'s
\texttt{NDSolve} function. We omit the quite lengthy details.
\end{ex}

Next we provide an example showing that in non-convex rectilinear polygons 
(i.\,e.\;other than a rectangle) evolution starting with continuous initial 
datum may develop discontinuities. 
\begin{ex}
 \label{excrack}
 Let
 \[\Omega = \{(x_1, x_2) \colon |(x_1, x_2)|_\infty \leq 1, x_1 \leq 0, x_2
\leq 0\}, \quad u_0(x_1, x_2) = x_2\]
 and so $\nabla u(0, \cdot) \equiv (0,1)$, $\z(0,\cdot) \equiv (0,1)$. The
solution can be written explicitly, for $t \leq \frac{1}{8}$ we have
\[ u(t, x_1, x_2) = \left\{\begin{array}{ll}
                         -1 + \sqrt{2t} & \text{if } x_2 \leq -1 + \sqrt{2t},
\\
                         -\sqrt{2t} & \text{if } x_1 \geq 0 \text{ and } x_2
\geq - \sqrt{2t}, \\
1 - \sqrt{2t} & \text{if } x_1 < 0 \text{ and } x_2 \geq 1-\sqrt{2t}, \\
x_2 & \text{otherwise.}
                        \end{array}\right.\]

We see that regions where $\nabla u = 0$ appear
near the boundary and expand with speed $\frac{1}{\sqrt{2t}}$. In these
regions, $z_2$ is a linear interpolation between $0$ and $1$. Also a jump in
the $x_2$ direction appears near $\bx = 0$ and grows with the same speed.
\end{ex}
\begin{figure}[h!]
    \centering
    \begin{subfigure}{0.3\textwidth}
        \includegraphics[width=\textwidth]{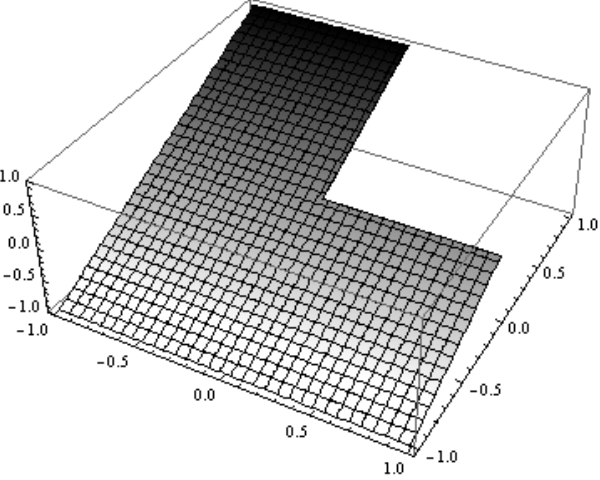}
        \caption{$t=0$.}
        \label{fig:excrackt0}
    \end{subfigure}
    \;
    \begin{subfigure}{0.3\textwidth}
        \includegraphics[width=\textwidth]{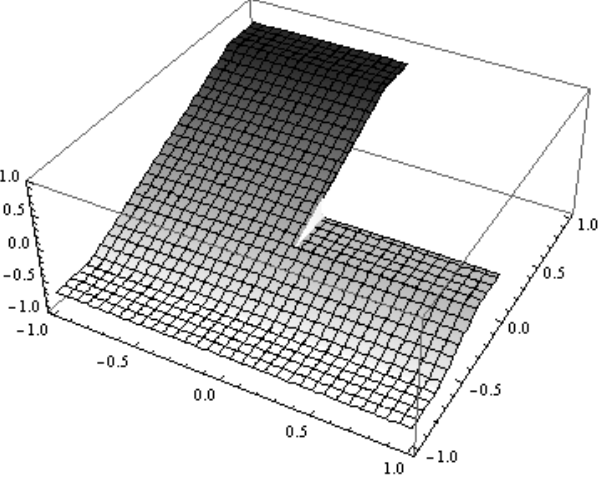}
        \caption{$t=0.04$.}
        \label{fig:excrackt0.04}
    \end{subfigure}
    \;
    \begin{subfigure}{0.3\textwidth}
        \includegraphics[width=\textwidth]{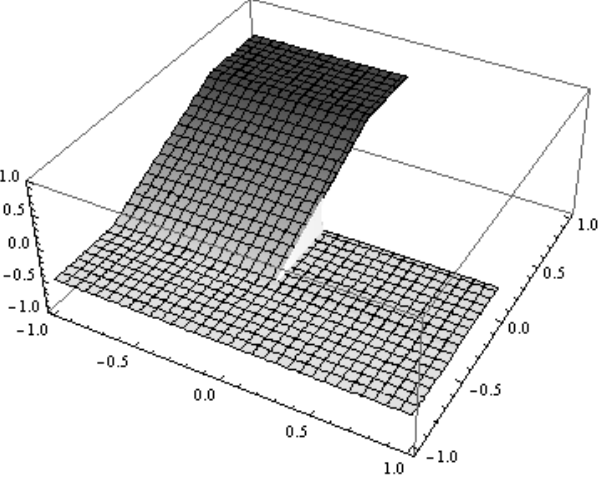}
        \caption{$t=0.12$.}
        \label{fig:excrackt0.12}
    \end{subfigure}
    \caption{Plots of $u(t,\cdot)$ from Example \ref{excrack} at certain time
instances $t$.}
\label{fig:excrack}
\end{figure}

Finally, let us present exact calculation for an example of the phenomenon of 
bending, which shows effectiveness of approximation with $PCR$ functions.
\begin{ex}\label{ex.1}
  Let $\Omega = \RR$, $u_0=\chi_{B_1(2)}$. We will show that the solution to
\eqref{1-tvflow} is given by $$u(t,\bx)=(1-v(\bx)t)_+,$$ with
$v=\frac{1}{2-\sqrt{2}}\chi_{B_\infty(\sqrt{2})\cap
B_1(2)}+\frac{1}{2-|\bx|_\infty}\chi_{B_1(2)\setminus B_\infty(\sqrt{2})}$.

 \smallskip In order to prove the claim, we approximate $B_1(2)$ by a
family of rectilinear polygons as follows. Given $n\in \mathbb N$, we define
inductively $$\left\{ \begin{array}{l} A_{1,1}:=B_\infty(1) \\
A_{k,1}:=B_\infty((\frac{2(k-1)}{n},0),1-\frac{k-1}{n})\setminus
A_{k-1,1} {\rm \ \ for \ }
k=2,\ldots,n\end{array}\right..$$ We observe that $\bigcup_{k=1}^n A_{k,1}$ is
an increasing sequence with respect to $n$ and that
$$\lim_{n\to\infty}\bigcup_{k=1}^n A_{k,1}=B_1(2)\cap\{x_1>0,-1\leq
x_2\leq 1\}.$$ By symmetry, we construct $A_{k,2}$, $A_{k,3}$ and $A_{k,4}$
for $k=1,\ldots, n$ such that  $$\lim_{n\to\infty}\bigcup_{k=1}^n
A_{k,2}=B_1(2)\cap\{x_2>0,-1\leq x_1\leq 1\},$$
 $$\lim_{n\to\infty}\bigcup_{k=1}^n A_{k,3}=B_1(2)\cap\{x_1<0,-1\leq
x_2\leq 1\},$$$$\lim_{n\to\infty}\bigcup_{k=1}^n
A_{k,4}=B_1(2)\cap\{x_2<0,-1\leq x_1\leq 1\}.$$ Therefore,
$B_1(2)=\lim_{n\to\infty} A_n:=\bigcup_{k=1}^n\bigcup_{j=1}^4 A_{k,j}.$

 \smallskip
 We let $C_k:=\bigcup_{j=1}^4 A_{k,1}$, for $k=1,\ldots,n$. Observe that the 
inequality $$\frac{{\rm Per}_1(C_1\cup\ldots \cup C_{k+1})}{|C_1\cup\ldots \cup 
C_{k+1})|}>\frac{{\rm Per}_1(C_1\cup\ldots \cup C_{k})}{|C_1\cup\ldots \cup 
C_{k})|}$$ holds (and therefore the facet $C_{k+1}$ breaks from $C_1\cup\ldots 
\cup C_{k})$ iff
 $$\frac{1}{1-\frac{k}{n}}>\frac{{\rm Per}_1(C_1\cup\ldots \cup 
C_{k})}{|C_1\cup\ldots \cup 
C_{k})|}=2\frac{1+\frac{k-1}{n}}{1+\frac{2(k-1)}{n}(1-\frac{k}{2n})}
\leftrightarrow \frac{k}{n}\geq 
\sqrt{\left(1-\frac{1}{2n}\right)^2+1}-\left(1-\frac{1}{2n}\right).$$

 Since the speed of  $C_j$ is  given by $\frac{1}{1-\frac{j}{n}}$ (which 
increases with respect to $j=k+1,\ldots,n$), once $C_{k+1}$ breaks from  
$C_1\cup\ldots \cup C_{k})$, so do $C_{j}$ from $C_{j-1}$ for $j=k+1,\ldots,n$.
 Therefore, the solution for $u_{0,n}=\chi_{A_n}$ is given by 
\begin{equation}\label{ex.1.appr.}u_n(t,\bx)=\left(1-\frac{{\rm 
Per}_1(C_1\cup\ldots \cup C_{k})}{|C_1\cup\ldots \cup 
C_{k})|}t\right)_+\chi_{(C_1\cup\ldots\cup 
C_k)}+\sum_{j=k+1}^n\left(1-\frac{1}{1-\frac{j}{n}}t\right)_+\chi_{C_j},
\end{equation} 
with $k\in\mathbb N$ satisfying $$[k-1]<  
n\left(\sqrt{\left(1-\frac{1}{2n}\right)^2+1}-\left(1-\frac{1}{2n}
\right)\right)\leq [k].$$ Letting $n\to\infty$ in \eqref{ex.1.appr.}, we finish 
the proof.

\end{ex}
\begin{figure}[h!]
    \centering
    \begin{subfigure}{0.3\textwidth}
        \includegraphics[width=\textwidth]{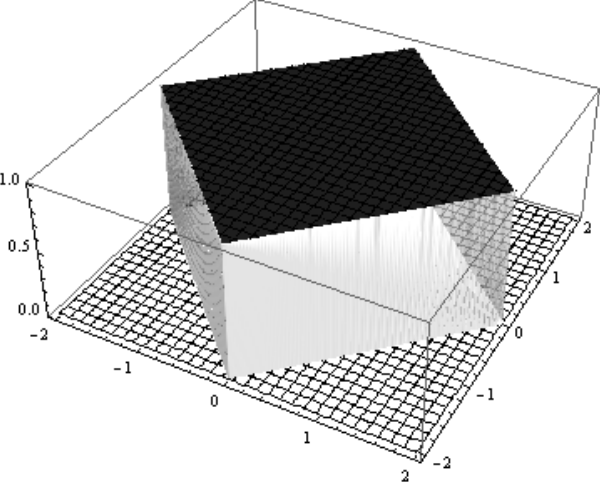}
        \caption{$t=0$.}
        \label{fig:ex1t0}
    \end{subfigure}
    \;
    \begin{subfigure}{0.3\textwidth}
        \includegraphics[width=\textwidth]{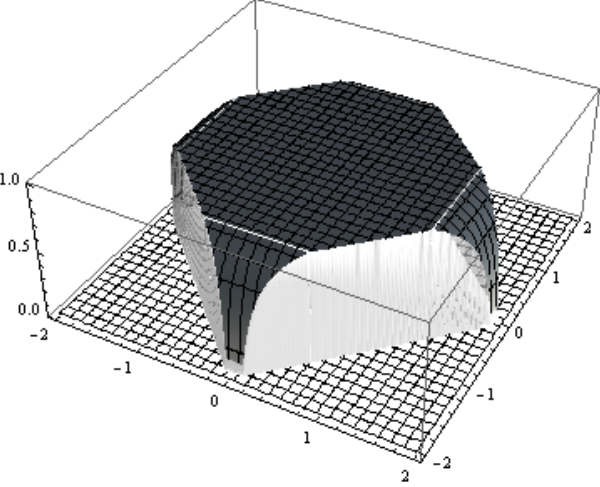}
        \caption{$t=0.1$.}
        \label{fig:ex1t0.1}
    \end{subfigure}
    \;
    \begin{subfigure}{0.3\textwidth}
        \includegraphics[width=\textwidth]{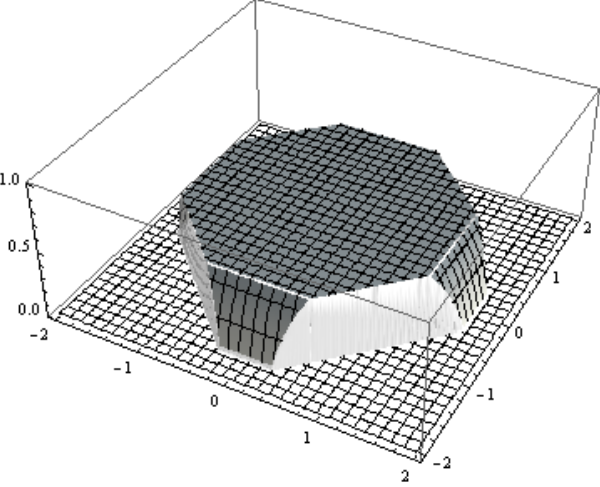}
        \caption{$t=0.3$.}
        \label{fig:ex1t0.3}
    \end{subfigure}
    \caption{Plots of $u(t,\cdot)$ from Example \ref{ex.1} at certain time
instances $t$.}
\label{fig:ex1}
\end{figure}
The evolution of a bounded convex domain $C$
satisfying an interior ball condition was explicitly given in \cite[Section
8.3]{cachamono}. There, the authors defined a notion of anisotropic variational 
mean
curvature, denoted by $H_C\colon \mathbb R^N\to ]-\infty,0]$, based on the
solvability of some auxiliary minimizing problems and on the existence of a
Cheeger set in $C$. Then, the solution to \eqref{atv} with data $u_0=\chi_C$ was
given by $$u(t,\bx)=(1+H_C(\bx)t)_+\chi_{C}(\bx).$$ In general, it is not 
obvious how to
compute this anisotropic variational mean curvature. However, Example \ref{ex.1}
shows that one can compute it by approximation of the set $C$ with rectilinear
polygons, even in the case that $C$ does not satisfy the interior ball
condition. Note that the solution
starting with initial datum as in Example \ref{ex.1} was calculated with an
approximate numerical procedure in \cite{muszkieta}. Here, we provided the exact
evolution in this case.

\section{Conclusions}
The core of our results is the explicit construction of
\emph{tetris-like solutions}, i.\,e.\;solutions in the
$PCR$ class. This class can be viewed as a natural generalization of
mono-dimensional step functions, whose finite dimensional structure
allows to effectively reduce the original nonlinear problem. The
directional diffusion allows to analyze the solutions only in a grid
given by a suitably chosen initial datum. We treat them as generic objects in
the set of all weak solutions, hence $PCR$ solutions are indeed
smooth functions in the new analytical language exclusively dedicated to our
variational problem. The detailed prescription of solutions allows to prove
even conservation of moduli of continuity for continuous initial data. It,
unexpected, removes this classical viewpoint on the issue of
solvability out of our interests. Just information obtained for $PCR$ functions
is much more complete than any knowledge of regularity in the classical setting.

At the end we would like to say a few words about the weakness of the approach.
The procedure works due to the possibility of introducing a grid. It
is the consequence of symmetry given by the $|\cdot|_1$ norm, the grid is just
determined by directions $\hat{e}_{x_1}$, $\hat{e}_{x_2}$ for the initial datum.
Here we have a natural shift symmetry, and the same structure at each vertex
of the grid. It seems that it would be possible to attempt to repeat at least
some of our analysis for anisotropic norms that generate a tiling of the
plane. Here we think of $|\cdot|_\phi$ determined by the hexagon, and tiling
given by honeycomb structure.
We are highly limited by regular tiling (triangular, rectangular and hexagonal),
however it seems to be possible to introduce more complex structure for
different anisotropy.
Such problems will definitely require new framework not linked to the classical analysis.

\section*{Acknowledgements}
The authors want to thank 
the anonymous referees whose comments really helped improve the quality of the 
paper. Thanks are also due to Micha{\l} Mi\' skiewicz for careful reading of 
the manuscript and pointing out some shortcomings. 

All figures were prepared using Wolfram Mathematica. 

The first
author has been supported by the grant of the National Science
Centre, Poland no.\;2014/13/N/ST1/02622. The second author acknowledges partial 
support by the Spanish MINECO and FEDER project
MTM2015-70227-P as well as the Simons Foundation
grant 346300 and the Polish Government MNiSW 2015-2019 matching 
fund.

\bibliographystyle{siamplain}
 \bibliography{aniso}
\end{document}